\documentclass[11pt]{amsart}
\usepackage{setspace}
\usepackage{amsmath,amssymb,mathrsfs,color}
\usepackage{verbatim}
\usepackage{cite}
\allowdisplaybreaks
\usepackage{hyperref}
\hypersetup{hypertex=true,
	colorlinks=true,
	linkcolor=blue,
	anchorcolor=blue,
	citecolor=blue}
\allowdisplaybreaks

\def\d{{\rm{d}}}

\def\N{{\mathbb N}}

\def\R{{\mathbb R}}
\def\C{{\mathbb C}}
\def\P{{\mathbb P}}

\newtheorem{lemma}{Lemma}[section]
\newtheorem{theorem}[lemma]{Theorem}
\newtheorem{remark}[lemma]{Remark}
\newtheorem{prop}[lemma]{Proposition}
\newtheorem{coro}[lemma]{Corollary}
\newtheorem{definition}[lemma]{Definition}
\newtheorem{example}[lemma]{Example}
\allowdisplaybreaks
\numberwithin{equation}{section}

\newtheorem{lemmaA}{Lemma}[section]
\newtheorem{lemmaB}{Lemma}[section]

\newcommand{\enabstractname}{Abstract}

\allowdisplaybreaks

\begin{document}
\title[Continuity of Lyapunov exponents for SDEs]{Continuity of Lyapunov exponents for stochastic differential equations}
	
\author{Zhenxin Liu}
\address{Z. Liu: School of Mathematical Sciences, Dalian University of Technology, Dalian 116024, P. R. China}
\email{zxliu@dlut.edu.cn}

\author{Lixin Zhang}
\address{L. Zhang (Corresponding author): School of Mathematical Sciences, Dalian University of Technology, Dalian 116024, P. R. China}
\email{lixinzhang8@hotmail.com; lixinzhang8@163.com}

\date{October 3, 2024}

\begin{abstract}
For non-autonomous linear stochastic differential equations (SDEs), we establish that the top Lyapunov exponent is continuous if the coefficients ``almost" uniformly converge. For autonomous SDEs, assuming the existence of invariant measures and the convergence of coefficients and their derivatives in pointwise sense, we get the continuity of all Lyapunov exponents. Furthermore, we demonstrate that for autonomous SDEs with strict monotonicity condition, all Lyapunov exponents are Lipschitz continuous with respect to the coefficients under the $L^{p,p}$ norm ($p>2$).  Similarly, the H\"older continuity of Lyapunov exponents holds under weaker regularity conditions. It seems that the continuity of Lyapunov exponents has not been studied for SDEs so far, in spite that there are many results in this direction for discrete-time dynamical systems.
\end{abstract}

\maketitle

\section{Introduction}
The Lyapunov exponent plays a critical role for characterizing the stability/chaos of dynamical systems. It provides intuitive insight into how trajectories respond to variations in initial conditions. Specifically, positive and negative Lyapunov exponents correspond to exponential expansion and contraction behaviors of trajectories, respectively. Consequently, investigations into the variations of Lyapunov exponents hold great significance.

The Lyapunov exponent was first introduced in the doctoral dissertation by Lyapunov \cite{Lyapunov1992} at the end of the 19th century. Consider the following $d$-dimensional autonomous ordinary differential equation (ODE)
$$
\d X= A(X)\d t,\quad X(0)=x\in\R^{d},
$$
where $A:\R^{d}\rightarrow\R^{d}$ is a general function with some property. The Lyapunov exponent of solution $\left\{\Phi(x,t)\right\}_{t\geq 0}$ along the vector $v\in\R^{d}$ is defined as
$$
\lambda(x,v):=\lim_{t\rightarrow\infty}\dfrac{1}{t}\ln\vert D\Phi(x,t) v\vert,
$$
where $D\Phi(x,t)$ is the Jacobian matrix of solution  $\left\{\Phi(x,t)\right\}_{t\geq 0}$. Furstenberg and Kesten \cite{Furstenberg1960}, Oseledets \cite{Oseledets1968} proved that the limit exists for almost every point with respect to an invariant measure, and the equation can yield a Lyapunov spectrum. Pesin \cite{Pesin} and Ruelle \cite{Ruelle} established the stable manifold theorem. These efforts solidly laid the foundation for the study of Lyapunov exponents, and one of the important research directions that emerged is the investigation of the continuity of Lyapunov exponents. In 1983, Ma$\tilde{\mathrm{n}}$$\acute{\mathrm{e}}$ \cite{Mane1984} emphasized the importance of the continuity of Lyapunov exponents and proposed a conjecture regarding their discontinuity for $C^{1}$ area--preserving diffeomorphisms.
According to Ma$\tilde{\mathrm{n}}$$\acute{\mathrm{e}}$'s manuscript, Bochi \cite{Bochi2002} presented the proof for this conjecture in his doctoral thesis. Furthermore, Bochi and Viana \cite{Bochi2005} provided a sufficient and necessary condition for the continuity of Lyapunov exponents in higher dimensions.   Many other important results, including Avila and Viana \cite{Avila}, Buzzi, Crovisior and Sarig \cite{Buzzi} and Viana and Young \cite{Viana}, have significantly contributed to the advancement of this field.

In the realm of random dynamical systems, we mention some studies focusing on Lyapunov exponents. Baxendale \cite{Baxendale} investigated the connection between the Lyapunov exponents of a stochastic flow on a manifold and its measure-preserving properties. Bahnm\"uller and Bogensch\"utz \cite{Bahnmuller} derived a version of the Margulis-Ruelle inequality specifically tailored for stochastic flows. In \cite{Liao97} and \cite{Liao00}, Liao provided a formula for Lyapunov exponents of stochastic flows and studied the relationship of Lyapunov exponents with the decomposition he provided. Additionally, Ledrappier and Young \cite{Ledrappier} presented the entropy formula for random transformations. There are many other important works in this field, such as Blumenthal, Xue and Yang \cite{Blumenthal2}, Kifer \cite{Kifer} and Liu and Qian \cite{Liupeidong}. These works have provided a good foundation for investigating Lyapunov exponents in the context of stochastic flows and random mappings.

Many physical systems are influenced by random factors, including environmental noise and measurement errors. SDEs effectively capture this randomness, allowing for a more accurate description of system behavior; for example, systems such as thermodynamics, gas dynamics, and so on. The Lyapunov exponent of SDEs can be used for stability analysis and long-term behavior prediction of the systems. This allows us to better understand the random effects within the systems and make reasonable decisions.

In this paper, we investigate the continuity of Lyapunov exponents for SDEs defined on $\R^{d}$.  Firstly, we consider the relatively simple case: the  non-autonomous linear SDE
\begin{equation}\label{equ.1.2}
\d X=A(t)X\d t+B(t)X\d W,
\end{equation}
where $W$ is a $1$-dimensional Brownian motion and $A,B:\R\rightarrow\R^{d\times d}$ are continuous functions bounded by a positive constant $K$. Equation \eqref{equ.1.2}  generates a fundamental solution matrix $\Phi_{0,t}(\omega)$ by Arnold \cite[p. 129]{Arnold-equation}. The Lyapunov exponent along the  vector $v$ for $\Phi_{0,t}(\omega)$ is defined as
$$
\lambda(v,\omega):=\limsup_{t\rightarrow\infty}\dfrac{1}{t}\ln\vert \Phi_{0,t}(\omega) v\vert.
$$
Cong \cite{Cong} proved the existence of this limit and is independent of $\omega\in\Omega$ for any  $v\in\R^{d}$. However, he primarily focused on conditions ensuring the sign-preserving of the top Lyapunov exponent $\lambda_{1}$ when $\lambda_{1}>0$.  In our results, we extend this analysis to cover all cases where $\lambda_{1}\in(-\infty,+\infty)$. And the required condition is weaker than their uniform convergence condition of the coefficients with respect to $t$. Additionally, we provide two examples to illustrate that our condition is sufficient but not necessary, and discuss the challenges in identifying a condition that is both sufficient and necessary.

Next,  we consider the general $d$-dimensional autonomous SDEs
\begin{equation}\label{equ.1.3}
\d X=A(X)\d t+B(X)\d W,\quad X(0)=x\in\R^{d}
\end{equation}
and
\begin{equation}\label{equ.1.5}
\d X_k=A_{k}(X_k)\d t+B_{k}(X_k)\d W, \quad X_k(0)=x, \quad k\in\N^{+},
\end{equation}
where $W$ is a $1$-dimensional Brownian motion and $A,B,A_{k},B_{k}: \R^{d}\rightarrow\R^{d}$ are $C^{1}$ functions.
Following Kunita \cite[p. 209]{Kunita}, we denote the stochastic flow $\Phi_{0,t}(x,\omega),t\geq 0$ generated by equation~\eqref{equ.1.3}  as $\Phi(x,t,\omega)$. The Lyapunov exponent  along the  vector $v$ at point $x$ for $\Phi(x,t,\omega)$ is defined by the limit
$$
\lambda(x,\omega,v):=\lim_{t\rightarrow\infty}\dfrac{1}{t}\ln\vert D\Phi(x,t,\omega) v\vert.
$$
It was proved that under some invariant measure $\mu$, for almost all $x\in\R^{d}$, $\P$-almost all $\omega\in\Omega$ and any $v\in\R^{d}$, the Lyapunov exponent exists and is  independent of $\omega\in\Omega$, and the stochastic flow can yield a Lyapunov spectrum; see Carverhill \cite{Carverhill} or Liu and Qian \cite[p. 119]{Liupeidong} for details. In our paper, we aim to identify conditions on the coefficients of equations \eqref{equ.1.3} and \eqref{equ.1.5} such that the Lyapunov exponents are continuous. To achieve this, we utilize the subadditive ergodic theorem, which allows us to relate the \text{long-term} behavior of solutions to time-$1$ behavior, i.e. \begin{equation}\label{equ.1.4}
\begin{aligned}
\int_{\R^{d}}	\lambda_{1}(x)\d \mu=&\int_{\R^{d}\times \Omega}\lim_{t\rightarrow\infty}\dfrac{1}{t}\ln\Vert D\Phi(x,t,\omega)\Vert\d \mu\times \P\\=&\int_{\R^{d}\times\Omega}\ln\Vert D\Phi(x,1,\omega) \Vert\d \mu\times \P,
\end{aligned}
\end{equation}
where $\lambda_{1}(x)$ is the top Lyapunov exponent. The formula \eqref{equ.1.4} is commonly known as the Furstenberg formula (see \cite[p. 102]{Viana-B}). In \cite{Furstenberg-C}, Furstenberg and Kifer employed this formula to establish the continuity of the top Lyapunov exponent  for random matrix products under an irreducibility assumption. If measures are not irreducible, Bocker-Neto and Viana \cite{Bocker-Viana} proved that the Lyapunov exponents of locally constant $\text{GL}(2,\C)$-cocycles over Bernoulli
shifts exhibit continuous variation with respect to both the cocycles and the measures.
So our goal has become to investigate the continuity of formula \eqref{equ.1.4} concerning the coefficients for equations \eqref{equ.1.3} and \eqref{equ.1.5}.

In general, the continuity of Lyapunov exponents is studied separately with respect to  either cocycles or invariant measures; for example Bochi and Viana \cite{Bochi2005} and Avila, Eskin and Viana \cite{Avila-E}. However, for SDEs, we consider both aspects simultaneously and provide a condition that depends only on the coefficients of equations. Unlike dominated splittings and irreducibility conditions, which are crucial for ensuring the continuity of Lyapunov exponents for cocycles, it seems that our condition is easier to verify. Indeed, we establish that for equations \eqref{equ.1.3} and \eqref{equ.1.5} with $C^{1}$ coefficients, whose derivatives are bounded and admit the same modulus of continuity, if the invariant measures converge weakly and the coefficients converge pointwise, then all Lyapunov exponents are continuous.

The H$\ddot{\mathrm{o}}$lder continuity of Lyapunov exponents has been a significant topic of discussions in recent years, for example, Cai et al \cite{Cai-You} and Duarte and Klein \cite{Duarte-K}, among others. Their works generally involve establishing the H\"older continuity of Lyapunov exponents with respect to cocycles in the $C^{0}$ topology. In our work, we prove the Lipschitz continuity of Lyapunov exponents with respect to coefficients of equations, and not in the $C^{0}$ norm. Specifically, under the strict monotonicity condition, the Lyapunov exponents exhibit Lipschitz continuity concerning the coefficients under the $L^{p,p}$ norm ($p>2$). Additionally, we demonstrate that the
$\alpha$-H$\ddot{\mathrm{o}}$lder  continuity ($0<\alpha<1$) of them also holds under weaker regularity conditions.
Furthermore, we develop a local Lipschitz continuity theory for Lyapunov exponents with respect to the coefficients in the space $C^{1,1}(\R^{d})$.

This paper is organized as follows. In section \ref{sec.2}, we introduce some preliminaries  related to stochastic differential equations and random dynamical systems. In section \ref{sec.3}, we prove the continuity of the top Lyapunov exponent for the non-autonomous linear equation \eqref{equ.1.2}.  Moving forward, section \ref{sec.4} establishes the continuity of all Lyapunov exponents for equation \eqref{equ.1.3}. Finally, in section \ref{sec.5}, we demonstrate that Lyapunov exponents are Lipschitz continuous or H$\ddot{\mathrm{o}}$lder  continuous with respect to the coefficients in equation \eqref{equ.1.3}.

\section{Preliminaries}\label{sec.2}
Firstly, let us introduce several norms. Let $(\Omega,\mathcal{F},\P)$ be a probability  space.  For any vector $v=(a_{1},a_{2},\dots,a_{d})\in \R^{d}$, we denote the Euclidean norm by $\vert v\vert:=\sqrt{\sum_{i=1}^{d}\vert a_{i}\vert^{2}}$. Given a matrix  $A=(v_{1},v_{2},\dots,v_{d})\in \R^{d\times d}$, the matrix norm $\vert A\vert$  is defined as $\vert A\vert:=\max\limits_{1\leq i\leq d}\left\{\vert v_{i}\vert\right\}$.  Furthermore, the operator norm $\Vert A\Vert$ is given by
$$
\Vert A\Vert:=\sup_{\substack{v\in \R^{d} \\ v\neq 0}}\frac{\vert Av\vert}{\vert v\vert}.
$$
  It is noteworthy that for any $A\in \R^{d\times d}$,
\begin{equation}\label{equ.2.4}
C_{1}^{-1}\vert A\vert\leq\Vert A\Vert\leq C_{1}\vert A\vert,
\end{equation}
where $C_{1}$ is a positive constant. We use $C^{1}(\R^{d})$ to  denote the space of continuously differentiable functions on $\R^{d}$ and $C_{b}^{1}(\R^{d})$ the space of functions that are bounded and have bounded derivatives. For any $\psi\in C_{b}^{1}(\R^{d})$, the norm is given by
$$
\Vert\psi\Vert_{C^{1}}:=\max\left\{\sup_{x\in \R^{d}}\left\{\vert\psi(x)\vert\right\}, \sup_{x\in \R^{d}}\left\{\vert D\psi(x)\vert\right\}\right\}.
$$
A function $f$ on $\R^{d}$ is called $C^{1,\alpha}$ if it is $1$-th continuously differentiable and the $1$-th derivatives
are H$\ddot{\mathrm{o}}$lder continuous of order $0<\alpha\leq1$.

Consider the SDE defined on $\R^{d}$ as follows:
\begin{equation}\label{equ.2.1}
\d X=A(X,t)\d t+B(X,t)\d W,\quad X(0)=x\in \R^{d},
\end{equation}
where $W$ is a $1$-dimensional Brownian motion, and $A,B:\R^{d}\times[0,+\infty)\rightarrow\R^{d}$.
It is well-known that when $A,B$ are  Lipschitz and satisfy the linear growth condition,  \eqref{equ.2.1} admits a unique solution; see e.g. Arnold \cite[p. 105]{Arnold-equation}.

Kunita \cite[p. 209]{Kunita} introduced the concept of two-parameter stochastic flows,  which capture the ``time-variable'' orbital motion. In this context,  we provide the definition of two-parameter
stochastic flows on the probability space $(\Omega,\mathcal{F},\P)$ with the state space $\R^{d}$. As we know, the solution of SDE \eqref{equ.2.1} corresponds to a two-parameter stochastic flow.

\begin{definition}[Stochastic flow]\label{def.2.1}\rm
A {\em  two-parameter stochastic flow of homeomorphisms} on $\R^{d}$ is a
family of continuous maps $\left\{\Phi_{s,t}(\omega)|\omega\in\Omega, s,t\geq 0 \right\}$ which satisfies the following properties for any $\omega$ from a subset $\Omega'\subset\Omega$ of full $\P$-measure:
	\begin{enumerate}
\item[1.]
$\Phi_{s,t}(\omega)=\Phi_{u,t}(\omega)\circ\Phi_{s,u}(\omega)$ holds for all $s<u<t$,
\item[2.]
$\Phi_{s,s}(\omega)$ is the identity map for all $s\geq 0$,
\item[3.]
the map $\Phi_{s,t}(\omega):\R^{d}\rightarrow \R^{d}$ is a  homeomorphism for all $s<t$.
	\end{enumerate}
\end{definition}

In this paper, we mainly consider two kinds of SDEs.
To ensure a comprehensive understanding,  we introduce some preliminaries separately for non-autonomous linear SDEs and autonomous SDEs.
\subsection{Non-autonomous Linear SDEs}
Consider the following $d$-dimensional non-autonomous linear SDE:
\begin{equation}\label{equ.2.9}
\d X=A(t)X\d t+B(t)X\d W,
\end{equation}
where $W$ is a $1$-dimensional Brownian motion and $A,B:\R\rightarrow \R^{d\times d}$ are continuous matrix-valued functions bounded by a constant $K$. The SDE \eqref{equ.2.9} generates a two-parameter stochastic flow $\left\{\Phi_{s,t}(\omega)| \omega\in\Omega, 0\leq s\leq t \right\}$ on $\R^{d}$, which correspond to linear operators on $\R^{d}$.

Next, we introduce the \textit{Lyapunov spectrum} with respect to stochastic flow $\Phi_{s,t}(\omega)$ by Cong~\cite{Cong}.

\begin{prop}\label{lya-linear}
 Under the assumption that $A,B$ are continuous functions bounded by a constant $K$, there exists a constant $p\leq d$ such that for almost all $\omega$, there exist a sequence $\left\{E^{s}_{i}(\omega)\subset \R^{d}| i=1,\dots,p\right\}$ of linear subspaces satisfying the following structure:
\begin{center}
$\R^{d}=E^{s}_{1}(\omega)\oplus E^{s}_{2}(\omega)\oplus\cdots\oplus E^{s}_{p}(\omega)$
\end{center}
and numbers independent of $\omega\in\Omega$
$$
\lambda_{1}>\lambda_{2}>\dots>\lambda_{p},
$$
such that for any $v\in E^{s}_{i}(\omega)$,
\begin{equation}\label{def-lya}
\limsup\limits_{t\rightarrow\infty}\frac{1}{t}\ln\vert \Phi_{s,t}(\omega)v\vert=:	\lambda_{i},\qquad 1\leq i\leq p.
\end{equation}
Furthermore, integer $d_{i}=\text{\rm dim}E^{s}_{i}(\omega)$ is the multiplicity of $\lambda_{i}$. Among them, the set of
$(\lambda_{i},d_{i})_{i=1,2,\ldots,p}$ is called {\em Lyapunov spectrum} of stochastic flow $\Phi_{s,t}(\omega)$ and $\{E_{1}^{s}(\omega),E_{2}^{s}(\omega),$ $\dots,E_{p}^{s}(\omega)\}$ is called {\em Oseledets splitting}. These components are invariant with respect to the orbits:
$$
\lambda_{i}(\Phi_{s,t}(\omega)v)=\lambda_{i}(v),\quad \Phi_{s,t}(\omega)E_{i}^{s}(\omega)=E_{i}^{t}(\omega), \quad 0\leq s\leq t, \quad 1\leq i\leq p.
$$
\end{prop}

\begin{remark}\rm	\begin{enumerate}
\item[1.]
While Lyapunov exponents of stochastic flow $\Phi_{s,t}(\omega)$ are deterministic and independent of $s$, the subspaces $E_{i}^{s}(\omega),1\leq i\leq p$ are generally random subspaces and depend on $s$.
\item[2.]   Lyapunov exponents obtained from continuous-time systems are equivalent to those obtained from discrete-time systems, as expressed by the equality:
$$
\limsup\limits_{t \rightarrow \infty}\dfrac{1}{t}\ln\vert\Phi_{0,t}(\omega)v\vert=\limsup\limits_{n \rightarrow \infty}\dfrac{1}{n}\ln\vert\Phi_{0,n}(\omega)v\vert, \quad t\in\R^{+},\quad n\in \N^{+},
$$
for almost all $\omega$ and any $v\in \R^{d}$.
\end{enumerate}
\end{remark}

\subsection{Autonomous SDEs}
Let us consider the $d$-dimensional autonomous SDE
\begin{equation}\label{equ.2.2}
\d X=A(X)\d t+B(X)\d W, \quad X(0)=x\in \R^{d},
\end{equation}
where $A,B: \R^{d}\rightarrow\R^{d}$ with properties to be specified below. Here, we denote the solution $\Phi_{0,t}(x,\omega)$ by $\Phi(x,t,\omega)$ or simply $\Phi(x,t)$. The following conclusion is provided by Kunita \cite[p. 218]{Kunita}.

\begin{prop}\label{lem.2.2}
Suppose that coefficients $A,B$ in SDE \eqref{equ.2.2} are $C^{1,\alpha}$ functions for some $0<\alpha\leq1$, and $\vert DA\vert,\vert DB\vert$ are bounded by a constant $L$. Then the solution $\Phi(x,t),t\geq 0$ is a $C^{1,\alpha}$ function of $x$ and for $1\leq l\leq d$ the derivative $\partial_{l}\Phi(x,t)=\frac{\partial \Phi(x,t)}{\partial x_{l}}$ satisfies the equation
\begin{equation*}
\partial_{l} \Phi(x,t)=e_{l}+\int_{0}^{t}DA(\Phi(x,s))\partial_{l} \Phi(x,s)\d s+\int_{0}^{t}DB(\Phi(x,s))\partial_{l} \Phi(x,s)\d W \quad a.s.,
\end{equation*}
where $DA(x),DB(x)$ are the Jacobian matrices of $A(x),B(x)$ and $e_{l}$ is the unit vector $(0,\ldots,0,$ $1,0,\ldots,0)$ with $1$ as the $l$-th component.
\end{prop}

Next, we introduce the definition of invariant measures for equation \eqref{equ.2.2} which play a crucial role in establishing the existence of Lyapunov exponents.
\begin{definition}[Invariant measure]\rm
A Borel probability measure on $\mu$ on $\R^{d}$ is called an {\em invariant measure} or {\em stationary distribution} of Markov process $\left\{\Phi(x,t)\right\}_{t\geq 0}$ if  $P_{t}^{\ast}\mu=\mu$ for all $t\in\R^{+}$, where $\left\{P_{t}^{\ast}\right\}_{t\geq0}$ is the Markov semigroup generated by $\left\{\Phi(x,t)\right\}_{t\geq 0}$.
\end{definition}

In this paper, we assume that equation \eqref{equ.2.2} possesses an invariant measure, denoted by $\mu$. Additionally, consider a family of $d$-dimensional autonomous SDEs
$$
\d X=A_{k}(X)\d t+B_{k}(X)\d W,\quad k\in\N^{+},
$$
where $A_{k},B_{k}: \R^{d}\rightarrow\R^{d}$. We denote the invariant measures corresponding to these SDEs by $\mu_{k},k\in\N^{+}$.

As is well-known, $\mu_{k}$ converges to $\mu$ under the $\text{weak}^{\ast}$-topology if for any $\phi\in C_{b}(\R^{d})$,
$$
\lim_{k\rightarrow\infty}\int_{\R^{d}}\phi(x)\d\mu_{k}=\int_{\R^{d}}\phi(x)\d\mu,
$$
where $C_{b}(\R^{d})$ is the set of all continuous bounded functions defined on $\R^{d}$.

For solution $\Phi(x,t),t\geq 0$ with invariant measure $\mu$,  the multiplicative ergodic theorem, as stated by Liu and Qian \cite[p. 119]{Liupeidong}, provides insights into its Lyapunov exponents.
\begin{prop}\label{lya-non}
If solution $\left\{\Phi(x,t,\omega)\right\}_{t\geq 0}$ has the following property
\begin{equation}\label{equ.2.3}
\int_{\R^{d}\times\Omega}\log^{+}\max\left\{\Vert D\Phi(x,1,\omega)\Vert,\Vert D\Phi^{-1}(x,1,\omega)\Vert\right\}\d\mu\times\P< +\infty,
\end{equation}
then
there exists a subset $\Lambda_{0}\subset \R^{d}\times\Omega$ with $\mu\times\P(\Lambda_{0})=1$ such that for any $(x,\omega)\in\Lambda_{0}$, there exists a sequence $\left\{ E^{i}(x,\omega)\right\}_{1\leq i\leq d}$ of linear subspaces of $\R^{d}$:
$$
\R^{d}=	E^{1}(x,\omega)\oplus\cdots\oplus E^{d}(x,\omega)
$$
and numbers independent of $\omega\in\Omega$
$$
\lambda_{1}(x)\geq\lambda_{2}(x)\geq\cdots\geq\lambda_{d}(x),
$$
such that for any $v\in E^{i}(x,\omega)$
$$	
\lim\limits_{t \rightarrow \infty}\dfrac{1}{t}\ln\vert D\Phi(x,t,\omega)v\vert=:\lambda_{i}(x),\quad 1\leq i\leq d.
$$
\end{prop}

\begin{remark} \rm
\begin{enumerate}
\item[(1)] If the coefficients of equation \eqref{equ.2.2} are $C^{1}$ functions and their derivatives are bounded, then the Jacobian matrix $D\Phi(x,1,\omega)$  for the equation satisfies the condition \eqref{equ.2.3} by Lemma \ref{lem.4.2} below.
\item[(2)]	If $\mu$ is an ergodic measure, then $ \lambda_{i}(x)$  is independent of $x$ for any $1\leq i\leq d$.
\item[(3)]  The Lyapunov exponents obtained from continuous-time systems are equivalent to those obtained from discrete-time systems, as expressed by the equality:
$$
\lim\limits_{t \rightarrow \infty}\dfrac{1}{t}\ln\vert\Phi(x,t,\omega)v\vert=\lim\limits_{n \rightarrow \infty}\dfrac{1}{n}\ln\vert\Phi(x,n,\omega)v\vert,\quad t\in\R^{+},\quad n\in \N^{+},
$$
for almost all $\omega$ and any $v\in \R^{d}$.
\item[(4)] Note that the Lyapunov exponent is defined by upper limit in Proposition \ref{lya-linear} for {\em non-autonomous} linear equations, instead of by limit in Proposition \ref{lya-non} for autonomous equations, because the limit may not exist in non-autonomous case even when the coefficients are smooth; see \cite{CLZ,CLZ-pre} for details.
\end{enumerate}
\end{remark}

Let $\wedge^{l}\R^{d}$ be the $l$-fold exterior power of $\R^{d}$, where $1\leq l\leq d$.  If $\left\{v_{j}:j=1,\dots,d\right\}$ is a basis of $\R^{d}$, then $\left\{v_{j_{1}}\wedge\dots\wedge v_{j_{l}}:1\leq j_{1}<\dots< j_{l}\leq d\right\}$ is a basis of $\wedge^{l}\R^{d}$. For any $v_{1}\wedge\dots\wedge v_{l}\in \wedge^{l}\R^{d}$, we define norm $\Vert v_{1}\wedge\dots\wedge v_{l}\Vert$ as the volume  $\text{vol}(v_{1},\dots,v_{l})$ of the corresponding parallelepiped. Suppose $A:\R^{d}\rightarrow\R^{d}$  is a linear operator. We can extend $A$ to act on the exterior power $\Lambda^{l}\R^{d}$ as follows:
$$
\wedge^{l}A(v_{1}\wedge\dots\wedge v_{l}):=Av_{1}\wedge\dots\wedge Av_{l}
$$
defines a linear operator $\wedge^{l}A$ on $\wedge^{l}\R^{d}$. Further, for any two linear operators $A,B:\R^{d}\rightarrow\R^{d}$, we have
\begin{equation}\label{equ.2.6}
\begin{aligned}
&\wedge^{l}(AB)=(\wedge^{l}A)(\wedge^{l}B), \qquad\qquad(\wedge^{l} A)^{-1}=\wedge^{l}A^{-1},\\
&\qquad\Vert\wedge^{l}A\Vert\leq\Vert A\Vert^{l}, \qquad\qquad\Vert\wedge^{l}AB\Vert\leq \Vert\wedge^{l}A\Vert\Vert\wedge^{l}B\Vert.
\end{aligned}
\end{equation}
Let us explicitly note here that
\begin{equation}\label{equ.2.5}
\sum_{i=1}^{l}\lambda_{i}(x)=\lim_{t\rightarrow\infty}\dfrac{1}{t}\ln\Vert\wedge^{l}D\Phi(x,t,\omega)\Vert \qquad \P-a.s.,
\end{equation}
as it is recorded in Viana \cite[p. 59]{Viana-B}.

\begin{prop}[Subadditive ergodic theorem; see Kingman \cite{Kingman}]\label{pro.2.1}
Let $(U,\mathcal{B},\nu)$ be a probability space and $T$ a measure-preserving transformation on $(U,\mathcal{B},\nu)$. Let $\left\{g_{n}\right\}^{+\infty}_{n=1}$ be a sequence of measurable functions $g_{n}:U\rightarrow\R\cup\left\{-\infty\right\}$ satisfying the conditions:
\begin{itemize}
\item  Integrability: $g_{1}^{+}\in L^{1}(U,\mathcal{B},\nu)$;
\item Subadditivity: $g_{m+n}\leq g_{m}+g_{n}\circ T^{m}\quad \nu-a.e.,\quad \text{ for all } m,n\geq 1.$
\end{itemize}
Then there exists a measurable function $g: U\rightarrow\R\cup\left\{-\infty\right\}$ such that
$$
g^{+}\in L^{1}(U,\mathcal{B},\nu),\quad g\circ T=g \quad \nu-a.e.,\quad \lim_{n\rightarrow+\infty}\dfrac{1}{n}g_{n}=g\quad\nu-a.e.
$$
and
$$
\lim_{n\rightarrow+\infty}\dfrac{1}{n}\int g_{n}\d\nu=\inf_{n}\dfrac{1}{n}\int g_{n}\d\nu=\int g\d\nu.
$$
\end{prop}

\begin{remark}\rm\label{rem.2.1}
Consider the solution $\Phi(x,t,\omega),t\geq 0$ with an invariant measure $\mu$ for equation \eqref{equ.2.2}. Denote $T(x,\omega):=(\Phi(x,1,\omega),\theta\omega)$. We know that $T$ is a measure-preserving transformation of $\mu\times\P$ on $\R^{d}\times\Omega$.
For the sequence $g_{n}(x,\omega)=\ln\Vert\wedge^{l} D\Phi(x,n,\omega)\Vert$, as $\Vert \wedge^{l}A_{1}A_{2}\Vert\leq \Vert \wedge^{l}A_{1}\Vert\Vert \wedge^{l}A_{2}\Vert$ for any $A_{1},A_{2}\in \text{GL}(d)$, the sequence is subadditive. Note that the function $g_{1}(x,\omega)=\ln \Vert \wedge^{l}D\Phi(x,1,\omega)\Vert$ is integrable, so
$$
\lim_{n\rightarrow+\infty}\dfrac{1}{n}\int_{\R^{d}\times \Omega}\ln \Vert \wedge^{l}D\Phi(x,n,\omega)\Vert\d\mu\times \P=\int_{\R^{d}\times \Omega} \ln\Vert \wedge^{l}D\Phi(x,1,\omega)\Vert \d\mu\times \P.
$$
\end{remark}

\section{Continuity of the top Lyapunov exponent for linear SDEs}\label{sec.3}
In this section, we consider the continuity of the top Lyapunov exponent for the $d$-dimensional linear SDE
\begin{equation}\label{equ.3.1}
\d X=A(t)X\d t+B(t) X\d W, \quad X(0)=x\in\R^{d},
\end{equation}
where the equation is approximated by the following $d$-dimensional equations
\begin{equation}\label{equ.3.2}
\d X_{k}=A_{k}(t)X_{k}\d t+B_{k}(t) X_{k}\d W,\quad X_{k}(0)=x\in\R^{d},\quad k\in\N^{+}.
\end{equation}
The coefficients $A,B,A_{k},B_{k}:\R^+\rightarrow \R^{d\times d}$ are continuous and bounded by the same constant $K$. Consequently, we obtain the corresponding stochastic flows $\Phi_{s,t}(\omega)$
and $\Phi_{k,s,t}(\omega)$, $0\leq s\leq t$. As discussed in Kunita \cite[p. 108]{Kunita2019}, the inverse flow $\Phi_{s,t}^{-1}$ is defined by the following backward SDE
$$
\Phi_{s,t}^{-1}(x)=x-\int_{s}^{t}[A(r)-B^{2}(r)]\Phi^{-1}_{r,t}(x)\d r-\int_{s}^{t} B(r)\Phi^{-1}_{r,t}(x)\d W(r),\quad x\in\R^{d}.
$$
The estimate of flow $\Phi^{-1}_{s,t}$ can be obtained using a similar method to that of Lemma~\ref{lem.3.1}.

\begin{lemma}\label{lem.3.1}
For equation \eqref{equ.3.1}, assuming $A,B$ are continuous and bounded by a constant $K$, we have the following estimate for any $0\leq s\leq t$,
$$
E\Vert\Phi_{s,t}(\omega)\Vert^{2}\leq Ce^{3K^{2}(t-s)^{2}+3K^{2}(t-s)},\quad 	E\Vert\Phi_{s,t}^{-1}(\omega)\Vert^{2}\leq Ce^{6(K^{2}+K^{4})(t-s)^{2}+3K^{2}(t-s)},
$$
where $C$ is a constant dependent only on $d$.
\end{lemma}

\begin{proof}
By the Cauchy-Schwarz inequality and It\^o's isometry,
$$
\begin{aligned}
E\vert\Phi_{s,t}(\omega)x\vert^{2}\leq& 3\vert x\vert^{2}+3E\big\vert\int_{s}^{t}A(r)\Phi_{s,r}(\omega)x\d r\big\vert^{2}+3E\big\vert\int_{s}^{t}B(r)\Phi_{s,r}(\omega)x\d W(r)\big\vert^{2}\\
\leq &3\vert x\vert^{2}+[3K^{2}(t-s)+3K^{2}]E\int_{s}^{t}\vert\Phi_{s,r}(\omega)x\vert^{2}\d r.
\end{aligned}
$$
Using the Gronwall inequality, we derive the estimate
$$
E\vert\Phi_{s,t}(\omega)x\vert^{2}\leq 3\vert x\vert^{2}e^{3K^{2}(t-s)^{2}+3K^{2}(t-s)},
$$
which, together with relationship \eqref{equ.2.4}, implies
$$
E\Vert\Phi_{s,t}(\omega)\Vert^{2}\leq C_{1}^{2}E\vert\Phi_{s,t}(\omega)\vert^{2}\leq \max_{\vert x\vert=1}E\vert\Phi_{s,t}(\omega)x\vert^{2}\leq 3C_{1}^{2}e^{3K^{2}(t-s)^{2}+3K^{2}(t-s)}.
$$
The proof of $\Phi_{s,t}$ is complete. Similarly, we can get estimate of $\Phi^{-1}_{s,t}$.
\end{proof}

\begin{remark}\label{rem-mom}\rm
It should be pointed out that we can get the estimate of any $p$-th ($p\ge 2$) moment of $\Phi_{s,t}(\omega)$ and $\Phi_{s,t}^{-1}(\omega)$
under the assumption of Lemma \ref{lem.3.1}. We omit details here.
\end{remark}

\begin{lemma}\label{lem.3.3}
For equations \eqref{equ.3.1} and \eqref{equ.3.2}, suppose that $A,B,A_{k},B_{k}$ are continuous and bounded by a constant $K$ for all $k\in\N^+$. If
$$
\lim\limits_{k\rightarrow+\infty}\sup_{t\in\R^{+}}\vert A_{k}(t)-A(t)\vert+\sup_{t\in\R^{+}}\vert B_{k}(t)-B(t)\vert=0,
$$
then
$$
\lim\limits_{k\rightarrow+\infty}E\Vert\Phi_{k,i-1,i}(\omega)-\Phi_{i-1,i}(\omega)\Vert^{2}=0\quad \text{uniformly},
$$
with respect ot $i\in\N^{+}$.
\end{lemma}

\begin{proof}
By relationship \eqref{equ.2.4}, we know that for any $k,i\in \N^{+}$
$$
\begin{aligned}
E\Vert\Phi_{k,i-1,i}(\omega)-\Phi_{i-1,i}(\omega)\Vert^{2}
\leq C_{1}^{2}E\vert\Phi_{k,i-1,i}(\omega)-\Phi_{i-1,i}(\omega)\vert^{2}.
\end{aligned}
$$
Following the method stated in Friedman \cite[p. 119]{Friedman}, for any $x\in\R^{d}$, we obtain the inequality
$$
\begin{aligned}
&E\vert\Phi_{k,i-1,i}(\omega)x-\Phi_{i-1,i}(\omega)x\vert^{2}\\
\leq& 3E\vert\eta_{k,i}\vert^{2}+3E\int_{i-1}^{i}\vert A_{k}(t)(\Phi_{k,i-1,t}(\omega)-\Phi_{i-1,t}(\omega))\vert^{2}\d t\\
&+3E\int_{i-1}^{i} \vert B_{k}(t)(\Phi_{k,i-1,t}(\omega)-\Phi_{i-1,t}(\omega))\vert^{2}\d t\\
\leq &3E\vert\eta_{k,i}\vert^{2}+6K^{2}E\int_{i-1}^{i} \vert\Phi_{k,i-1,t}(\omega)-\Phi_{i-1,t}(\omega)\vert^{2}\d t,
\end{aligned}
$$
where
$$
\eta_{k,i}=\int_{i-1}^{i} (A_{k}(t)-A(t))\Phi_{i-1,t}(\omega)\d t+\int_{i-1}^{i} (B_{k}(t)-B(t))\Phi_{i-1,t}(\omega)\d W.
$$
The Gronwall inequality implies that
\begin{equation}\label{equ.3.12}
\begin{aligned}
E\vert\Phi&_{k,i-1,i}(\omega)x-\Phi_{i-1,i}(\omega)x\vert^{2}
\leq 3e^{6K^{2}}E\vert\eta_{k,i}\vert^{2}\\
\leq& 6e^{6K^{2}}\big[\sup_{t\in[i-1,1]}\vert A_{k}(t)-A(t)\vert^{2}+\sup_{t\in[i-1,1]}\vert B_{k}(t)-B(t)\vert^{2}\big]E\int_{i-1}^{i}\vert\Phi_{i-1,t}(\omega)\vert^{2}\d t.
\end{aligned}
\end{equation}
According to Lemma \ref{lem.3.1}, we have
$$
E\int_{i-1}^{i}\vert\Phi_{i-1,t}(\omega)\vert^{2}\d t\leq C,$$
where $C$ is independent of $i\in\N^{+}$. With this, the proof of the lemma is complete.
\end{proof}

Before stating the continuity theorem, it is essential to introduce the following property of sequences $\left\{A_{k}\right\}_{k\in\N^{+}}$ and $\left\{B_{k}\right\}_{k\in\N^{+}}$.

\noindent\textbf{Property 1.}
{\em For any $k\in \N^{+}$, there exists a sequence $\left\{U_{k,i} \right\}_{i\in \N^{+}}$, where $U_{k,i}$ consists of at most countably many disjoint intervals within interval $[i-1,i)$. And the sequence $\left\{U_{k,i} \right\}_{i\in \N^{+}}$ satisfies
\begin{equation}\label{equ.3.9}
\lim_{k\rightarrow\infty}\limsup\limits_{n\rightarrow\infty}\dfrac{1}{n}\sum_{i=1}^{n} (L(U_{k,i}))^{\frac{1}{2}}=0,
\end{equation}
where $L(U_{k,i})$ is the sum of the lengths of all the intervals it contains. For any given $k\in\N^{+}$, let
\[
A'_k(t):= \left\{\begin{array}{ll}
		A(t),	& t\in U_{k,i},\\
		A_k(t),	& t\in [i-1,i)\setminus U_{k,i},
\end{array}\right.
\quad \hbox{for all }i\in \N^+.
\]
The sequence $\left\{A'_{k}\right\}_{k\in\N^{+}}$ converges uniformly to $A$ on $\R^+$ as $k\rightarrow\infty$. The similar holds for the sequence $\left\{B_{k}\right\}_{k\in\N^{+}}$.}

Property 1 means that we can change $A_k$ and $B_k$ over a relatively ``small" set such that they uniformly converge on $\R^+$.  Note that if $A_k\to A$ and $B_k\to B$ uniformly on $\R^+$, then Property 1 trivially holds.

\begin{theorem}\label{the.3.1}
For equations \eqref{equ.3.1} and \eqref{equ.3.2}, if Property 1 holds and $A,B,A_{k},B_{k},k\in \N^{+}$ are continuous functions bounded by a constant $K$, then the top Lyapunov exponent is continuous:
\begin{center}
$\lim\limits_{k\rightarrow\infty}\lambda_{1,k}=\lambda_{1}$.
\end{center}
\end{theorem}

\begin{proof}
For simplicity, we denote the stochastic flows $\Phi_{0,t}$ and $\Phi_{k,0,t}$ of equations~\eqref{equ.3.1} and \eqref{equ.3.2} as $\Phi(t)$ and $\Phi_{k}(t)$ respectively, for any $t\geq 0$ and $k\in \N^{+}$. Then the top Lyapunov exponents can be expressed by the formulas
\begin{center}$
\limsup\limits_{t\rightarrow\infty}\dfrac{1}{t}\ln \Vert \Phi(t) \Vert=\lambda_{1},\qquad
\limsup\limits_{t\rightarrow\infty}\dfrac{1}{t}\ln \Vert \Phi_{k}(t) \Vert=\lambda_{1,k}.$
\end{center}
\textbf{(1)}
Firstly, let us consider the situation that the sequences $\left\{A_{k}\right\}_{k\in\N^{+}}$ and $\left\{B_{k}\right\}_{k\in\N^{+}}$ converge uniformly to $A$ and $B$ on $\R^+$.

Note that
\begin{equation}\label{equ.3.3}
\begin{aligned}
\left| \lambda_{1,k}-\lambda_{1}\right|=&\Big\vert\limsup\limits_{t\rightarrow\infty}\dfrac{1}{t}\ln \Vert\Phi_{k}(t) \Vert-	\limsup\limits_{t\rightarrow\infty}\dfrac{1}{t}\ln \Vert\Phi(t) \Vert\Big\vert\\
=&\Big\vert\limsup\limits_{n\rightarrow\infty}\dfrac{1}{n}\ln \Vert\Phi_{k}(n) \Vert-	\limsup\limits_{n\rightarrow\infty}\dfrac{1}{n}\ln \Vert\Phi(n) \Vert\Big\vert\\
\leq&\Big\vert	\limsup\limits_{n\rightarrow\infty}\dfrac{1}{n}\ln \dfrac{\Vert\Phi_{k}(n)  \Vert}{\Vert\Phi(n)  \Vert}\Big\vert\\
=&\max\left\{\limsup\limits_{n\rightarrow\infty}\dfrac{1}{n}\ln \dfrac{\Vert\Phi(n)  \Vert}{\Vert\Phi_{k}(n)  \Vert},\limsup\limits_{n\rightarrow\infty}\dfrac{1}{n}\ln \dfrac{\Vert\Phi_{k}(n)  \Vert}{\Vert\Phi(n)  \Vert}\right\}.\\
\end{aligned}
\end{equation}
Here, the first term can be estimated as follows:
$$
\limsup\limits_{n\rightarrow\infty}\dfrac{1}{n}\ln \dfrac{\Vert\Phi(n)  \Vert}{\Vert\Phi_{k}(n)  \Vert}
\leq \limsup\limits_{n\rightarrow\infty}\dfrac{1}{n}\ln \Vert\Phi(n)\Phi_{k}^{-1}(n) \Vert.
$$
As we know, the stochastic flows $\Phi(n)$ and $\Phi_{k}(n)$ can be expressed as $\Phi_{n-1,n}\Phi_{n-2,n-1}\cdots\Phi_{0,1}$ and $\Phi_{k,n-1,n}\Phi_{k,n-2,n-1}\cdots\Phi_{k,0,1}$. For simplicity,  let $\Phi_{i-1,i}$ and $\Phi_{k,i-1,i}$ be denoted as $\xi_{i}$ and $\xi_{k,i}$.  By applying the property $\Vert AB\Vert=\Vert BA\Vert$ for invertible matrices $A$ and $B$, we can deduce that
$$
\begin{aligned}
\Vert\Phi(n)\Phi_{k}^{-1}(n)\Vert=&\Vert\xi_{n}\xi_{n-1}\cdots\xi_{1}(\xi_{k,n}\xi_{k,n-1}\cdots\xi_{k,1})^{-1}  \Vert\\
=&\Vert\xi_{n}\xi_{n-1}\cdots\xi_{1}\xi_{k,1}^{-1}\cdots\xi_{k,n-1}^{-1}\xi_{k,n}^{-1}\Vert\\
=&\Vert\xi_{k,n}^{-1}\xi_{n}\xi_{n-1}\cdots\xi_{1}\xi_{k,1}^{-1}\cdots\xi_{k,n-1}^{-1}\Vert\\
\leq&\Vert\xi_{k,n}^{-1}\xi_{n}\Vert\Vert\xi_{n-1}\cdots\xi_{1}\xi_{k,1}^{-1}\cdots\xi_{k,n-1}^{-1}\Vert.\\
\end{aligned}
$$
Through induction, we derive the following estimate
$$
\Vert\Phi(n)\Phi_{k}^{-1}(n)  \Vert\leq\prod_{i=1}^{n}\Vert\xi_{i}\xi_{k,i}^{-1}\Vert.
$$
So we get
\begin{equation}\label{equ.3.7}
\begin{aligned}
\limsup\limits_{n\rightarrow\infty}\dfrac{1}{n}\ln \Vert\Phi(n)\Phi_{k}^{-1}(n)  \Vert\leq\limsup\limits_{n\rightarrow\infty}\dfrac{1}{n}\sum_{i=1}^{n}\ln\Vert\xi_{i}\xi_{k,i}^{-1}\Vert.
\end{aligned}
\end{equation}

Denote $\ln^{+}x:=\max\left\{0,\ln x\right\}$ and $ \ln^{-}x:=\max\left\{0,-\ln x\right\}.$
Consider the $4$-th moment
$$
E\big\vert \ln^{+}\Vert\xi_{i}\xi_{k,i}^{-1}\Vert-E\ln^{+}\Vert\xi_{i}\xi_{k,i}^{-1}\Vert\big\vert^{4}\leq 8E\big\vert\ln^{+}\Vert\xi_{i}\xi_{k,i}^{-1}\Vert\big\vert^{4}+8\big\vert E\ln^{+}\Vert\xi_{i}\xi_{k,i}^{-1}\Vert\big\vert^{4},
$$
where for any $k,i\in\N^{+}$,
$$
E\big\vert\ln^{+}\Vert\xi_{i}\xi_{k,i}^{-1}\Vert\big\vert^{4}\leq E\Vert\xi_{i}\xi_{k,i}^{-1}\Vert^{4}\leq (E\Vert\xi_{i}\Vert^{8})^{\frac{1}{2}}\times (E\Vert\xi_{k,i}^{-1}\Vert^{8})^{\frac{1}{2}}\leq C
$$
by Remark \ref{rem-mom}. Hence, the random variables $\ln^{+}\Vert\xi_{i}\xi_{k,i}^{-1}\Vert-E\ln^{+}\Vert\xi_{i}\xi_{k,i}^{-1}\Vert,k,i\in\N^{+}$ have bounded $4$-th moments.
 By the strong law of large numbers \cite[p. 388]{Cantelli}, the equality holds
\begin{equation}\label{equ.3.27}
\lim\limits_{n\rightarrow\infty}\dfrac{1}{n}\Big[\sum_{i=1}^{n}\ln^{+}\Vert\xi_{i}\xi_{k,i}^{-1}\Vert-\sum_{i=1}^{n}E(\ln^{+}\Vert\xi_{i}\xi_{k,i}^{-1}\Vert)\Big]=0.\qquad (\P-a.s.)
\end{equation}
Similarly,
$$
E\big\vert \ln^{-}\Vert\xi_{i}\xi_{k,i}^{-1}\Vert-E\ln^{-}\Vert\xi_{i}\xi_{k,i}^{-1}\Vert\big\vert^{4}\leq 8E\big\vert\ln^{-}\Vert\xi_{i}\xi_{k,i}^{-1}\Vert\big\vert^{4}+8\big\vert E\ln^{-}\Vert\xi_{i}\xi_{k,i}^{-1}\Vert\big\vert^{4}
$$
and
$$\begin{aligned}
&E\big\vert\ln^{-}\Vert\xi_{i}\xi_{k,i}^{-1}\Vert\big\vert^{4}=E\big\vert\ln^{+}\Vert\xi_{i}\xi_{k,i}^{-1}\Vert^{-1}\big\vert^{4}\leq E\Vert\xi_{i}\xi_{k,i}^{-1}\Vert^{-4}\\
&\qquad\qquad\qquad\qquad\qquad\qquad\qquad\leq E\Vert\xi_{k,i}\xi_{i}^{-1}\Vert^{4}\leq (E\Vert\xi_{k,i}\Vert^{8})^{\frac{1}{2}}\times (E\Vert\xi_{i}^{-1}\Vert^{8})^{\frac{1}{2}}\leq C.
\end{aligned}
$$
So we have
\begin{equation}\label{equ.3.33}
\lim\limits_{n\rightarrow\infty}\dfrac{1}{n}\Big[\sum_{i=1}^{n}\ln^{-}\Vert\xi_{i}\xi_{k,i}^{-1}\Vert-\sum_{i=1}^{n}E(\ln^{-}\Vert\xi_{i}\xi_{k,i}^{-1}\Vert)\Big]=0.\qquad (\P-a.s.)
\end{equation}
According to equalities \eqref{equ.3.27} and \eqref{equ.3.33}, we have
\begin{equation}\label{mome}
\lim\limits_{n\rightarrow\infty}\dfrac{1}{n}\Big[\sum_{i=1}^{n}\ln\Vert\xi_{i}\xi_{k,i}^{-1}\Vert-\sum_{i=1}^{n}E(\ln\Vert\xi_{i}\xi_{k,i}^{-1}\Vert)\Big]=0.\qquad (\P-a.s.)
\end{equation}

Note that
\begin{align}\label{equ.3.13}
&\limsup\limits_{n\rightarrow\infty}\dfrac{1}{n}\sum_{i=1}^{n}E(\ln\Vert\xi_{i}\xi_{k,i}^{-1}\Vert)\nonumber\\
\leq&\limsup\limits_{n\rightarrow\infty}\dfrac{1}{n}\sum_{i=1}^{n}E\big(\ln(1+\Vert(\xi_{k,i}-\xi_{i})\xi_{k,i}^{-1}\Vert)\big)\nonumber\\
\leq&\limsup\limits_{n\rightarrow\infty}\dfrac{1}{n}\sum_{i=1}^{n}E\big(\Vert(\xi_{k,i}-\xi_{i})\Vert\Vert\xi_{k,i}^{-1}\Vert\big)\nonumber\\
\leq&\limsup\limits_{n\rightarrow\infty}\dfrac{1}{n}\sum_{i=1}^{n}\big(E\Vert\xi_{k,i}-\xi_{i}\Vert^{2}\big)^{\frac{1}{2}}
\big(E\Vert\xi_{k,i}^{-1}\Vert^{2}\big)^{\frac{1}{2}}.
\end{align}
Taking the limit as $k\rightarrow\infty$ on both sides of the above inequality, we have
$$
\begin{aligned}
\lim\limits_{k\rightarrow\infty}\limsup\limits_{n\rightarrow\infty}\dfrac{1}{n}\sum_{i=1}^{n}E(\ln\Vert\xi_{i}\xi_{k,i}^{-1}\Vert)
\leq \lim_{k\rightarrow\infty}\limsup\limits_{n\rightarrow\infty}\dfrac{1}{n}\sum_{i=1}^{n}\big(E\Vert\xi_{k,i}-\xi_{i}\Vert^{2}\big)^{\frac{1}{2}}\big(E\Vert\xi_{k,i}^{-1}\Vert^{2}\big)^{\frac{1}{2}}.
\end{aligned}
$$
Since
$$
\lim_{k\rightarrow\infty}	E\Vert\xi_{k,i}-\xi_{i}\Vert^{2}=0
$$
uniformly with respect to $i\in \N^{+}$ by Lemma \ref{lem.3.3} and $E\Vert\xi_{k,i}^{-1}\Vert^{2},k,i\in\N^{+}$ share the same bound by Lemma \ref{lem.3.1}, it follows that
\begin{equation}\label{equ.3.4}
\begin{aligned}
\lim\limits_{k\rightarrow\infty}\limsup\limits_{n\rightarrow\infty}\dfrac{1}{n}\sum_{i=1}^{n}E(\ln\Vert\xi_{i}\xi_{k,i}^{-1}\Vert)
\leq0.
\end{aligned}
\end{equation}

According to estimates \eqref{equ.3.7}, \eqref{mome} and \eqref{equ.3.4}, we obtain
$$
\begin{aligned}
\lim_{k\rightarrow\infty}\limsup\limits_{n\rightarrow\infty}\dfrac{1}{n}\ln \dfrac{\Vert\Phi(n)  \Vert}{\Vert\Phi_{k}(n)  \Vert}\leq0.
\end{aligned}
$$
By the similar method,
$$
\lim_{k\rightarrow\infty}\limsup\limits_{n\rightarrow\infty}\dfrac{1}{n}\ln \dfrac{\Vert\Phi_{k}(n)  \Vert}{\Vert\Phi(n)  \Vert}\leq0.
$$
Then
$$
\lim_{k\rightarrow\infty}\vert\lambda_{1,k}-\lambda_{1}\vert=0.
$$
So far, the proof of \textbf{(1)} is complete.

\textbf{(2)}
Next, our objective is to demonstrate that if the sequences $\left\{A_{k}\right\}_{k\in \N^{+}}$ and $\left\{B_{k}\right\}_{k\in \N^{+}}$ satisfy the Property 1, then the top Lyapunov exponent is continuous.

From estimates \eqref{equ.3.7}, \eqref{mome} and \eqref{equ.3.13} in \textbf{(1)}, we get
\begin{align}\label{equ.3.17}
\limsup\limits_{n\rightarrow\infty}\dfrac{1}{n}\ln \dfrac{\Vert\Phi(n)  \Vert}{\Vert\Phi_{k}(n)  \Vert}
\leq\limsup\limits_{n\rightarrow\infty}\dfrac{1}{n}\sum_{i=1}^{n}\big(E\Vert\xi_{k,i}-\xi_{i}\Vert^{2}\big)^{\frac{1}{2}}\big(E\Vert\xi_{k,i}^{-1}\Vert^{2}\big)^{\frac{1}{2}}.
\end{align}
According to inequality \eqref{equ.3.12} in Lemma \ref{lem.3.3}, we know that for some constant $C$,
\begin{equation}\label{equ.3.15}
\begin{aligned}
E\Vert\xi_{k,i}-\xi_{i}\Vert^{2}
\leq CE\vert\eta_{k,i}\vert^{2},\quad k,i\in\N^{+},
\end{aligned}
\end{equation}
where
$$
\eta_{k,i}=\int_{i-1}^{i} (A_{k}(t)-A(t))\Phi_{i-1,t}(\omega)\d t+\int_{i-1}^{i} (B_{k}(t)-B(t))\Phi_{i-1,t}(\omega)\d W.
$$
Note that
$$
\begin{aligned}
E\vert\eta_{k,i}\vert^{2}\leq& 2E\Big\vert\int_{i-1}^{i} (A_{k}(t)-A(t))\Phi_{i-1,t}(\omega)\d t\Big\vert^{2}+ 2E\Big\vert\int_{i-1}^{i} (B_{k}(t)-B(t))\Phi_{i-1,t}(\omega)\d W\Big\vert^{2}\\
\leq&2E\int_{i-1}^{i} \vert(A_{k}(t)-A(t))\Phi_{i-1,t}(\omega)\vert^{2}\d t+2E\int_{i-1}^{i} \vert(B_{k}(t)-B(t))\Phi_{i-1,t}(\omega)\vert^{2}\d t\\
\leq &2E\int_{[i-1,i]\setminus U_{k,i}} \vert(A_{k}(t)-A(t))\Phi_{i-1,t}(\omega)\vert^{2}\d t+2E\int_{[i-1,i]\setminus U_{k,i}} \vert(B_{k}(t)-B(t))\Phi_{i-1,t}(\omega)\vert^{2}\d t\\
&+2E\int_{U_{k,i}} \vert(A_{k}(t)-A(t))\Phi_{i-1,t}(\omega)\vert^{2}\d t+2E\int_{U_{k,i}} \vert(B_{k}(t)-B(t))\Phi_{i-1,t}(\omega)\vert^{2}\d t,\\
\end{aligned}
$$
where $U_{k,i}$ is defined in Property 1. We denote the first two terms of the above inequality as $\xi'_{k,i}$ and the last two terms as $\xi''_{k,i}$. By \eqref{equ.3.15} and Lemma \ref{lem.3.1},
\begin{align}\label{xi}
\limsup\limits_{n\rightarrow\infty}\dfrac{1}{n}\sum_{i=1}^{n}\big(E\Vert\xi_{k,i}-\xi_{i}\Vert^{2}\big)^{\frac{1}{2}}\big(E\Vert\xi_{k,i}^{-1}\Vert^{2}\big)^{\frac{1}{2}}
\leq C\limsup\limits_{n\rightarrow\infty}\dfrac{1}{n}\sum_{i=1}^{n}(\xi'_{k,i})^{\frac{1}{2}}+(\xi''_{k,i})^{\frac{1}{2}}.
\end{align}

According to the Property 1,
$$
\xi'_{k,i}=2E\int_{i-1}^{i} \vert(A'_{k}(t)-A(t))\Phi_{i-1,t}(\omega)\vert^{2}\d t+2E\int_{i-1}^{i} \vert(B'_{k}(t)-B(t))\Phi_{i-1,t}(\omega)\vert^{2}\d t.
$$
Given that $A'_{k}$ and $B'_{k}$ uniformly converge to $A$ and $B$, it is not difficult to get
$$
\lim_{k\rightarrow\infty}\limsup\limits_{n\rightarrow\infty}\dfrac{1}{n}\sum_{i=1}^{n}(\xi'_{k,i})^{\frac{1}{2}}=0.
$$
Furthermore, by the boundedness of $\vert A\vert,\vert B\vert,\vert A_{k}\vert,\vert B_{k}\vert$ and Lemma \ref{lem.3.1},
$$
\begin{aligned}
\xi''_{k,i}=& 2E\int_{U_{k,i}} \vert(A_{k}(t)-A(t))\Phi_{i-1,t}(\omega)\vert^{2}\d t+2E\int_{U_{k,i}} \vert(B_{k}(t)-B(t))\Phi_{i-1,t}(\omega)\vert^{2}\d t\\
\leq&16K^{2}\int_{U_{k,i}} E\vert\Phi_{i-1,t}(\omega)\vert^{2}\d t\\
\leq &C L(U_{k,i}),
\end{aligned}
$$
where $C$ is a constant dependent only on $d$. Therefore, by equality \eqref{equ.3.9},
\begin{align}
\lim_{k\rightarrow\infty}\limsup\limits_{n\rightarrow\infty}\dfrac{1}{n}\sum_{i=1}^{n}(\xi''_{k,i})^{\frac{1}{2}}
\leq  C\lim_{k\rightarrow\infty}\limsup\limits_{n\rightarrow\infty}\dfrac{1}{n}\sum_{i=1}^{n} (L(U_{k,i}))^{\frac{1}{2}}
=0.\notag
\end{align}
It follows from \eqref{equ.3.17} and \eqref{xi} that
$$
\lim_{k\rightarrow\infty}\limsup\limits_{n\rightarrow\infty}\dfrac{1}{n}\ln \dfrac{\Vert\Phi(n)  \Vert}{\Vert\Phi_{k}(n)  \Vert}\leq0.
$$

Similarly, we can get
$$
\lim_{k\rightarrow\infty}\limsup\limits_{n\rightarrow\infty}\dfrac{1}{n}\ln \dfrac{\Vert\Phi_{k}(n)  \Vert}{\Vert\Phi(n)  \Vert}\leq0.
$$
Therefore, we have by \eqref{equ.3.3}
$$
\lim_{k\rightarrow\infty}\vert\lambda_{1,k}-\lambda_{1}\vert=0.
$$
The proof is complete.
\end{proof}

\begin{remark}\rm
If the stochastic flows $\Phi_{s,t}(\omega)$ and $\Phi_{k,s,t}(\omega)$ are Lyapunov regular, which means that the limits \eqref{def-lya} (instead of upper limits) in the definition of Lyapunov exponents exist,
then the sum of $l$ largest Lyapunov exponents (counted with multiplicity) can be repressed as follows:
$$
\sum_{i=1}^{l}\lambda_{i}=\lim_{n\rightarrow \infty}\frac{1}{t}\ln\Vert \wedge^{l}\Phi_{0,t}(\omega)\Vert \qquad \P-a.s., \qquad 1\leq l\leq d.
$$
In this case, using the same method, we can prove that all Lyapunov exponents are continuous under the same condition.
\end{remark}

For a better understanding of the Property 1, consider an example that the sequence $\left\{A_{k}\right\}_{k\in\N^{+}}$ does not uniformly converge to $A$, but satisfies the Property 1, which implies that the top Lyapunov exponent $\lambda_{1,k}$ converges to $\lambda_{1}$. To simplify, we omit the step of smoothing the coefficients.

\begin{example}\rm
Consider the following $1$-dimensional linear  ODEs
$$
	\begin{aligned}
	\d X=A(t)X\d t,	\qquad \d X=A_{k}(t)X\d t,\quad k\in\N^{+},
	\end{aligned}
$$
where $A$ and $A_{k}$ are given by the formulas
$$
	A(t)=1,\quad t\in[0,+\infty),\quad
	\begin{aligned}
		A_{k}(t)=\left\{\begin{array}{ll}
		1,	& t\in [0,k),\\
			2,	& t\in [k,k+1),\\
		1,	& t\in [k+1,+\infty),
		\end{array}\right.
		\end{aligned}
\quad  k\in\N^{+}.
$$
Clearly, $A_{k}$ dose not converge to $A$ uniformly.
But the sequence $\left\{A_{k}\right\}_{k\in \N^{+}}$ satisfies the Property~1. And the top Lyapunov exponent is continuous:
$$
\lambda_{1,k}=\lambda_{1}=1 \quad \text{for all } k\in \N^{+}.
$$
\end{example}

 Next, we give another example that the top Lyapunov exponent is continuous, but the coefficients do not satisfy the Property 1. This example demonstrates that the condition in Theorem \ref{the.3.1} is a sufficient but not necessary condition. We think it is challenging to derive a sufficient and necessary condition.

\begin{example}\rm
Consider the $2$-dimensional linear diagonal equations
$$
\begin{aligned}
	&\d X=A(t)X\d t=\mathrm{diag}[a_{1}(t), a_{2}(t)]X\d t, \\
	&	\d X=A_{k}(t)X\d t=\mathrm{diag}[a_{k,1}(t), a_{k,2}(t)]X\d t, \quad k\in\N^{+}.
\end{aligned}
$$
The coefficients are defined by the following formulas:
\begin{align}
& a_{1}(t)=1, \quad t\in[0,+\infty),\quad  &&a_{2}(t)=6, \quad t\in[0,+\infty),\notag\\
	&a_{k,1}(t)=\left\{\begin{array}{ll}
	1,	& t\in [0,k),\\
	2,	 &t\in [k,+\infty),
\end{array}\right. \quad  &&a_{k,2}(t)=6, \quad t\in[0,+\infty),\quad  k\in\N^{+}.\notag
\end{align}
So we get the top Lyapunov exponents $\lambda_{1,k}=\lambda_{1}=6$ for all $k\in \N^{+}$. However, we observe that the sequence $\left\{A_{k}\right\}_{k\in \N^{+}}$ does not satisfy Property 1.
\end{example}

\section{Continuity of Lyapunov exponents for autonomous SDEs}\label{sec.4}
In this section, we consider the $d$-dimensional autonomous SDE
\begin{equation}\label{equ.4.7}
	\d X=A(X)\d t+B(X)\d W, \quad X(0)=x\in \R^{d},\\
\end{equation}
whose coefficients are approximated by sequences $\left\{A_{k}\right\}_{k\in \N^{+}}$ and $\left\{B_{k}\right\}_{k\in\N^{+}}$ respectively. They constitute the $d$-dimensional SDEs
	\begin{equation}\label{equ.4.8}
	\d X=A_{k}(X)\d t+B_{k}(X)\d W,\quad X_{k}(0)=x\in \R^{d}, \quad k\in\N^{+}.\\
\end{equation}
Denote solutions $\Phi(x,t)$ and $\Phi_{k}(x,t)$ for equations \eqref{equ.4.7} and \eqref{equ.4.8}, respectively. Our objective is to determine the conditions  under which Lyapunov exponents of the approximating systems \eqref{equ.4.8} converge to those of the original system \eqref{equ.4.7}.

As described in Proposition \ref{lem.2.2}, the derivative $\partial_{l} \Phi(x,t)$ of the solution to \eqref{equ.4.7} satisfies the following equation
\begin{equation}\label{equ.4.37}
	\partial_{l} \Phi(x,t)=e_{l}+\int_{0}^{t}DA(\Phi(x,s))\partial_{l} \Phi(x,s)\d s+\int_{0}^{t}DB(\Phi(x,s))\partial_{l} \Phi(x,s)\d W, \quad 1\leq l\leq d.
\end{equation}
As discussed in Kunita \cite[p. 228]{Kunita}, the inverse $(D\Phi)^{-1}(x,t)$ of the Jacobian $D\Phi(x,t)$ satisfies the following equation
$$
\begin{aligned}
(D \Phi)^{-1}(x,t)=I-&\int_{0}^{t}DA(\Phi(x,s))(D \Phi)^{-1}(x,s)\d s-\int_{0}^{t}DB(\Phi(x,s))(D \Phi)^{-1}(x,s)\d W\\
-&\int_{0}^{t}(DB(\Phi(x,s)))^{2}(D \Phi)^{-1}(x,s)\d s,
\end{aligned}
$$
where $I$ is the identity matrix.

In what follows, for simplicity constant $C$ may change from line to line. The following result is elementary, so we omit the proof.

\begin{lemma}\label{lem.4.2}
For equation \eqref{equ.4.7}, let coefficients $A,B$ belong to $ C^{1}(\R^{d})$, $\vert DA\vert,\vert DB\vert$ be bounded by a constant $K$.
Then for any $p \geq 2,x,y\in\R^{d}$ and $t\in[0,T]$ with $T\in[0,+\infty)$, the solution $\Phi(x,t)$ satisfies
$$
E\vert \Phi(x,t)-\Phi(y,t)\vert^{p}\leq C\vert x-y\vert^{p};
$$
the Jacobian $D\Phi=(\partial_{1}\Phi,\partial_{2}\Phi,\dots,\partial_{d}\Phi)$ and its inverse $(D\Phi)^{-1}$ exist and satisfy for $x\in\R^{d}$ and $t\in [0,T]$ that
$$
E\vert D\Phi(x,t)\vert^{p}\leq C, \quad E\vert (D \Phi)^{-1}(x,t)\vert^{p}\leq  C.
$$
Here $C=C(p,K,T)>0$ is a constant dependent only on $p,K,T$.
\end{lemma}

\begin{theorem}\label{the.4.1}
For equations \eqref{equ.4.7} and \eqref{equ.4.8}, let coefficients $A,B,A_{k},B_{k}$ belong to $ C^{1}(\R^{d})$, $DA,DB, DA_{k}, DB_{k}$ be bounded by a constant $K$ and admit the same modulus $m$ of continuity. Suppose that there exist ergodic measures $\mu$ and $\mu_{k}$ of equations \eqref{equ.4.7} and \eqref{equ.4.8} respectively and
$$
\mu_{k}\rightarrow \mu \qquad \text{as } k\rightarrow+\infty,
$$
under the $\text{weak}^{\ast}$-topology.
If
\begin{equation}\label{equ.4.17}
	\lim\limits_{k\rightarrow\infty}\vert A_{k}(x)-A(x)\vert+\vert B_{k}(x)-B(x)\vert=0 \quad \text{point-wise}
\end{equation}
and
\begin{equation}\label{equ.4.16}
	\lim\limits_{k\rightarrow\infty}\vert DA_{k}(x)-DA(x)\vert+\vert DB_{k}(x)-DB(x)\vert=0 \quad \text{point-wise},
\end{equation}
then for any $1\leq i\leq d$,
\begin{center}
	$\lim\limits_{k\rightarrow\infty}\lambda_{i,k}=\lambda_{i},$
\end{center}
where $\lambda_{i}$ and $\lambda_{i,k}$ are Lyapunov exponents corresponding to $\mu$ and $\mu_{k}$ respectively.
\end{theorem}

\begin{proof}
For simplicity, we identify $X$ as $\R^{d}$ and denote $\Phi(x,\omega)=\Phi(x,1,\omega)$ and $\Phi_{k}(x,\omega)=\Phi_{k}(x,1,\omega)$, respectively.
By \eqref{equ.2.5} and Remark \ref{rem.2.1}, we know that
 $$\sum_{i=1}^{l}\lambda_{i}=\int_{X}\int_{\Omega}\ln\Vert\wedge^{l} (D\Phi(x,\omega))\Vert\d\P\d\mu, \quad \sum_{i=1}^{l}\lambda_{i,k}=\int_{X}\int_{\Omega}\ln\Vert\wedge^{l} (D\Phi_{k}(x,\omega))\Vert\d\P\d\mu_{k}.$$
Then we have
\begin{align}\label{term.4}
&\vert\sum_{i=1}^{l}\lambda_{i,k}-\sum_{i=1}^{l}\lambda_{i}\vert\notag\\
=&\Big\vert\int_{X}\int_{\Omega}\ln\Vert\wedge^{l} D\Phi_{k}(x)\Vert\d\P\d\mu_{k}-\int_{X}\int_{\Omega}\ln\Vert\wedge^{l} D\Phi(x)\Vert\d\P\d\mu\Big\vert\notag\\
	\leq&\Big\vert\int_{X}\int_{\Omega}\ln\Vert\wedge^{l} D\Phi_{k}(x)\Vert\d \P\d\mu_{k}-\int_{X}\int_{\Omega}\ln\Vert\wedge^{l} D\Phi(x)\Vert\d \P\d\mu_{k}\Big\vert\\
	&+\Big\vert
\int_{X}\int_{\Omega}\ln\Vert\wedge^{l} D\Phi(x)\Vert\d \P\d\mu_{k}-\int_{X}\int_{\Omega}\ln\Vert \wedge^{l}D\Phi(x)\Vert\d \P\d\mu\Big\vert.\notag
\end{align}
By Lemma \ref{lem.4.2}, there exists some constant $C$ such that for any $x\in X$,
$$
\big\vert E\ln\Vert\wedge^{l} D\Phi(x)\Vert\big\vert\leq C.
$$
Since $\mu_{k}$ converges to $\mu$ under the $\text{weak}^{\ast}$-topology, the second term on the right-hand side of \eqref{term.4}
converges to zero as $k\rightarrow\infty$. So to complete the proof, we only need to show that the first term of \eqref{term.4}
goes to zero.

For the first term of \eqref{term.4}, note that
\begin{equation}\label{equ.4.4}
\begin{aligned}
&\Big\vert	\int_{X}\int_{\Omega}\ln\Vert\wedge^{l} D\Phi_{k}(x)\Vert\d \P\d\mu_{k}-	\int_{X}\int_{\Omega}\ln\Vert\wedge^{l} D\Phi(x)\Vert\d \P\d\mu_{k}\Big\vert\\
\leq & \max\Big\{\int_{X}\int_{\Omega}\ln\Vert\wedge^{l} D\Phi(x)\Vert\d \P\d\mu_{k}-	\int_{X}\int_{\Omega}\ln\Vert\wedge^{l} D\Phi_{k}(x)\Vert\d \P\d\mu_{k},\\
&
\qquad\qquad\qquad\int_{X}\int_{\Omega}\ln\Vert\wedge^{l} D\Phi_{k}(x)\Vert\d \P\d\mu_{k}-	\int_{X}\int_{\Omega}\ln\Vert\wedge^{l} D\Phi(x)\Vert\d \P\d\mu_{k}\Big\}.
\end{aligned}
\end{equation}
Consider the first one:
\begin{align}\label{equ.4.9}
&	\int_{X}\int_{\Omega}\ln\Vert\wedge^{l} D\Phi\Vert\d \P\d\mu_{k}-	\int_{X}\int_{\Omega}\ln\Vert\wedge^{l} D\Phi_{k}\Vert\d \P\d\mu_{k}\notag\\
\leq&\int_{X}\int_{\Omega}\ln\Vert\wedge^{l}( D\Phi(D\Phi_{k})^{-1})\Vert\d \P\d\mu_{k}\notag\\
\leq &l\times\int_{X}\int_{\Omega}\ln\Vert D\Phi(D\Phi_{k})^{-1}\Vert\d \P\d\mu_{k}\notag\\
= &l \times\int_{X}\int_{\Omega}\ln\Vert [D\Phi-D\Phi_{k}+D\Phi_{k}](D\Phi_{k})^{-1}\Vert\d \P\d\mu_{k}\notag\\
\leq &l \times\int_{X}\int_{\Omega}\Vert [D\Phi_{k}-D\Phi](D\Phi_{k})^{-1}\Vert\d\P \d\mu_{k}\notag\\
\leq &C\times \Big( \int_{X}\int_{\Omega}\vert D\Phi_{k}-D\Phi\vert^{2}\d \P\d\mu_{k}\Big)^{\frac{1}{2}}\times \Big( \int_{X}\int_{\Omega}\Vert (D\Phi_{k})^{-1}\Vert^{2}\d \P\d\mu_{k}\Big)^{\frac{1}{2}}
\end{align}
holds by \eqref{equ.2.6} and \eqref{equ.2.4}.

Next, in order to prove that
\begin{equation}\label{equ.4.26}
\lim\limits_{k\rightarrow+\infty} \Big\vert\int_{X}\int_{\Omega}\vert D\Phi_{k}-D\Phi\vert^{2}\d \P\d\mu_{k}-\int_{X}\int_{\Omega}\vert D\Phi_{k}-D\Phi\vert^{2}\d \P\d\mu\Big\vert=0,
\end{equation}
we denote
$$
\phi_{k}(x) :=E\vert D\Phi_{k}(x)-D\Phi(x)\vert^{2}.
$$
For any $\epsilon>0$, there exists a compact subset $X_{\epsilon}\subset X$ such that $\mu(X\setminus X_{\epsilon})<\epsilon$.
Choose a subsequence $\left\{{k_{l}}\right\}_{l\in\N^{+}}$ such that
\begin{equation}\label{subseq}
\limsup_{k\rightarrow\infty}\int_{ X_{\epsilon}}\phi_{k}\d(\mu_{k}-\mu)=\lim\limits_{l\rightarrow\infty}\int_{ X_{\epsilon}}\phi_{k_{l}}\d(\mu_{k_{l}}-\mu).
\end{equation}
On the one hand, note that for any $x,y\in X$,
\begin{align}\label{equ.4.12}
\vert\phi_{k}(x)-\phi_{k}(y)\vert \leq &  E\Big(\vert D\Phi_{k}(x)-D\Phi(x)-D\Phi_{k}(y)+D\Phi(y)\vert\notag\\
&\qquad\times \big[\vert D\Phi_{k}(x)-D\Phi(x)+ D\Phi_{k}(y)-D\Phi(y)\vert\big] \Big)\notag\\
\leq &\big(E\vert D\Phi_{k}(x)-D\Phi(x)-D\Phi_{k}(y)+D\Phi(y)\vert^{2}\big)^{\frac{1}{2}}\notag\\
&\qquad\times \sqrt{2}\big( E\vert D\Phi_{k}(x)-D\Phi(x)\vert^{2}+E\vert D\Phi_{k}(y)-D\Phi(y)\vert^{2}\big)^{\frac{1}{2}}\notag\\
\leq&2\sqrt{2C}\big(E\vert D\Phi_{k}(x)-D\Phi_{k}(y)\vert^{2}+E\vert D\Phi(x)-D\Phi(y)\vert^{2}\big)^{\frac{1}{2}},
\end{align}
where for any $x\in X$,
\begin{equation}\label{equ.4.19}
E\vert D\Phi_{k}(x)-D\Phi(x)\vert^{2}\leq 2E\vert D\Phi_{k}(x)\vert^{2}+2E\vert D\Phi(x)\vert^{2}\leq C
\end{equation}
by Lemma \ref{lem.4.2}. $D\Phi$ is a matrix-valued function which can be expressed by $(\partial_{1}\Phi,\partial_{2}\Phi,\dots,\partial_{d}\Phi)$, where $\partial_{l}\Phi$ satisfies equation \eqref{equ.4.37} for any $1\leq l\leq d$. By the method stated in Friedman \cite[p. 119]{Friedman}, notice that
\begin{equation}\label{equ.4.23}
E\vert \partial_{l}\Phi_{k}(x)-\partial_{l}\Phi_{k}(y)\vert^{2}\leq C E\vert\eta_{k}(1)\vert^{2},
\end{equation}
where
$$\begin{aligned}
\eta_{k}(1)=&\int_{0}^{1}[DA_{k}(\Phi_{k}(x,t))-DA_{k}(\Phi_{k}(y,t))]\partial_{l}\Phi_{k}(y,t)\d t \\
&+\int_{0}^{1}[DB_{k}(\Phi_{k}(x,t))-DB_{k}(\Phi_{k}(y,t))]\partial_{l}\Phi_{k}(y,t)\d W.
\end{aligned}
$$
By the Cauchy-Schwarz inequality,  It\^o's isometry and Lemma \ref{lem.4.2}, we get
\begin{equation}\label{equ.4.15}
\begin{aligned}	
E\vert\eta_{k}(1)\vert^{2}\leq&C\Big[\big(E\int_{0}^{1}\vert DA_{k}(\Phi_{k}(x,t))-DA_{k}(\Phi_{k}(y,t))\vert^{4}\d t\big)^{\frac{1}{2}}\\
&+\big(E\int_{0}^{1}\vert DB_{k}(\Phi_{k}(x,t))-DB_{k}(\Phi_{k}(y,t))\vert^{4}\d t\big)^{\frac{1}{2}}\Big].
\end{aligned}
\end{equation}
The inequality, that for any $\epsilon>0$, $k\in\N^{+}$ and $t\in[0,1]$,
$$
\begin{aligned}
\P(\vert\Phi_{k}(x,t)-\Phi_{k}(y,t)\vert>\epsilon)\leq& \dfrac{1}{\epsilon^{2}}E\vert \Phi_{k}(x,t)-\Phi_{k}(y,t)\vert^{2}
\leq\dfrac{C}{\epsilon^{2}}\vert x-y\vert^{2}
\end{aligned}
$$
holds by the Chebyshev inequality and Lemma~\ref{lem.4.2}, where $C$ is a constant independent of $k$. Along with $\lim_{x\rightarrow 0}m(x)=0$, it follows that for above $\epsilon$, there exists some constant $\delta>0$ such that when $\vert x-y\vert<\delta$, there exists a subset $\Omega_{k,t,\epsilon}\subset \Omega$ satisfying   $\P(\Omega\setminus\Omega_{k,t,\epsilon})<\epsilon$ on which
$$
m\big(\vert \Phi_{k}(x,t,\omega)-\Phi_{k}(y,t,\omega)\vert\big)<\epsilon.
$$
Then we get
\begin{align}
&E\int_{0}^{1}\vert DA_{k}(\Phi_{k}(x,t,\omega))-DA_{k}(\Phi_{k}(y,t,\omega))\vert^{4}\d t\notag\\
\leq&\int_{0}^{1} \int_{\Omega_{k,t,\epsilon}}\big[m\big(\vert \Phi_{k}(x,t,\omega)-\Phi_{k}(y,t,\omega)\vert\big)\big]^{4}\d\P\d t\notag\\
&\qquad+\int_{0}^{1}\int_{\Omega\setminus\Omega_{k,t,\epsilon}}\vert DA_{k}(\Phi_{k}(x,t,\omega))-DA_{k}(\Phi_{k}(y,t,\omega))\vert^{4}\d\P\d t\notag\\
<& \epsilon^{4}+(2K)^{4}\epsilon,\notag
\end{align}
and similarly
\[
E\int_{0}^{1}\vert DB_{k}(\Phi_{k}(x,t))-DB_{k}(\Phi_{k}(y,t))\vert^{4}\d t \leq \epsilon^{4}+(2K)^{4}\epsilon.
\]
By \eqref{equ.4.23} and \eqref{equ.4.15}, we know that for any $\epsilon'>0$, there exists $\delta>0$ such that when $\vert x-y\vert<\delta$, the first term of \eqref{equ.4.12} satisfies
$$
E\vert D\Phi_{k}(x)-D\Phi_{k}(y)\vert^{2}< \epsilon'\quad \hbox{for all } k\in\N^+.
$$
By applying the same method, we can also show that the second term has a similar estimate.
So $\phi_{k},k\in\N^{+}$ are equicontinuous on $X$.

On the other hand, by Lemma \ref{lem.4.2}, $\phi_{k}(x),k\in\N^{+}$ are uniformly bounded. In terms of the Arzel\`a-Ascoli theorem, we can extract a sub-subsequence uniformly converging to $0$ on the compact subset $X_{\epsilon}\subset X$. Relabeling the original subsequence
in \eqref{subseq} if necessary, we get
$$
\limsup_{k\rightarrow\infty}\int_{ X_{\epsilon}}\phi_{k}\d(\mu_{k}-\mu)=\lim\limits_{l\rightarrow\infty}\int_{ X_{\epsilon}}\phi_{k_{l}}\d(\mu_{k_{l}}-\mu)=0.
$$
And by the same method,
$$
\limsup_{k\rightarrow\infty}\int_{ X_{\epsilon}}\phi_{k}\d(\mu-\mu_{k})=0.
$$
That is,
$$\lim_{k\rightarrow\infty} \int_{ X_{\epsilon}}\phi_{k}\d(\mu_{k}-\mu)=0.$$
So when $k$ is large enough,
$$
\begin{aligned}
&\Big\vert\int_{X}\phi_k(x)\d\mu_{k}-\int_{X}\phi_k(x)\d\mu\Big\vert \\
\leq &\Big\vert\int_{ X_{\epsilon}}\phi_{k}\d(\mu_{k}-\mu)\Big\vert+\Big\vert\int_{X\setminus X_{\epsilon}}\phi_k \d\mu_{k}-\int_{X\setminus X_{\epsilon}}\phi_k \d\mu\Big\vert\\
< & \epsilon+2C\epsilon
\end{aligned}
$$
holds by \eqref{equ.4.19} and the fact that $\mu_{k}\rightarrow\mu$ under the $\text{weak}^{\ast}$-topology. So far, we get the desired result \eqref{equ.4.26}.

Note that by the method stated in Friedman \cite[p. 119]{Friedman}, for any $1\leq l\leq d$, there exists a constant $C=C(K)$ dependent on $K$ such that
\begin{equation}\label{equ.4.14}
\begin{aligned}
&\int_{X}E\vert \partial_{l}\Phi_{k}(x)-\partial_{l}\Phi(x)\vert^{2}\d\mu\\
\leq& C\int_{X}E\Big\vert\int_{0}^{1}[DA_{k}(\Phi_{k}(x,t))-DA(\Phi(x,t))]\partial_{l}\Phi(x,t)\d t\Big\vert^{2}\d\mu\\
& +C\int_{X}E\Big\vert\int_{0}^{1}[DB_{k}(\Phi_{k}(x,t))-DB(\Phi(x,t))]\partial_{l}\Phi(x,t)\d W\Big\vert^{2}\d\mu.
\end{aligned}
\end{equation}
For the first term on the right-hand side of \eqref{equ.4.14}, the following estimate holds by the H\"older inequality:
\begin{align}\label{equ.4.13}
&\int_{X}	E\Big\vert\int_{0}^{1}[DA_{k}(\Phi_{k}(x,t))-DA(\Phi(x,t))]\partial_{l}\Phi(x,t)\d t\Big\vert^{2}\d\mu\notag \\
&\leq 2\int_{X}\int_{0}^{1}E\big\vert[DA_{k}(\Phi_{k}(x,t))-DA_{k}(\Phi(x,t))]\partial_{l}\Phi(x,t)\big\vert^{2} \d t\d\mu\notag \\
&\quad+2\int_{X}\int_{0}^{1}E\big\vert[DA_{k}(\Phi(x,t))-DA(\Phi(x,t))]\partial_{l}\Phi(x,t)\big\vert^{2} \d t\d\mu\notag \\
&\leq 2\Big( \int_{X}\int_{0}^{1}E\vert  DA_{k}(\Phi_{k}(x,t))-DA_{k}(\Phi(x,t))\vert^{p} \d t\d\mu\Big) ^{\frac{2}{p}}\cdot\Big(\int_{X} \int_{0}^{1}E\vert \partial_{l}\Phi(x,t)\vert^{q} \d t\d\mu\Big) ^{\frac{2}{q}}\\
&\quad+2\Big( \int_{X}\int_{0}^{1}E\vert DA_{k}(\Phi(x,t))-DA(\Phi(x,t))\vert^{p} \d t\d\mu\Big) ^{\frac{2}{p}}\cdot\Big(\int_{X} \int_{0}^{1}E\vert \partial_{l}\Phi(x,t) \vert^{q} \d t\d\mu\Big) ^{\frac{2}{q}},\notag
\end{align}
where $p>2$ and $\frac{2}{p}+\frac{2}{q}=1$. The Lebesgue dominated convergence theorem implies that
\begin{equation}\label{equ.4.1}
\begin{aligned}
&\lim\limits_{k\rightarrow\infty}\int_{X}\int_{0}^{1}E\vert DA_{k}(\Phi_{k}(x,t))-DA_{k}(\Phi(x,t))\vert^{p} \d t\d\mu\\	=&\int_{X}\int_{0}^{1}\lim_{k\rightarrow\infty}E\vert DA_{k}(\Phi_{k}(x,t))-DA_{k}(\Phi(x,t))\vert^{p} \d t\d\mu.
\end{aligned}
\end{equation}
Due to \eqref{equ.4.17}, for any $x\in X$ and $t\in[0,1]$,
$\Phi_{k}(x,t)$ converges to $\Phi(x,t)$ in probability, as $k\rightarrow\infty$. So for any $\epsilon>0$, there exists a constant $K_{1}$ such that for any $k>K_{1}$ we can find a subset $\Omega_{k,\epsilon}\subset \Omega$ satisfying $\P(\Omega\setminus\Omega_{k,\epsilon})<\frac{\epsilon}{2^{p+1}K^{p}}$ on which
$$
m(\vert\Phi_{k}(x,t,\omega)-\Phi(x,t,\omega)\vert)<(\dfrac{\epsilon}{2})^{\frac{1}{p}}.
$$
It follows that
\begin{align}
	&E\big\vert DA_{k}(\Phi_{k}(x,t))-DA_{k}(\Phi(x,t))\big\vert^{p}\notag\\
	= &\int_{\Omega_{k,\epsilon}}\big\vert DA_{k}(\Phi_{k}(x,t))-DA_{k}(\Phi(x,t))\big\vert^{p}\d \P+\int_{\Omega\setminus\Omega_{k,\epsilon}}\big\vert DA_{k}(\Phi_{k}(x,t))-DA_{k}(\Phi(x,t))\big\vert^{p}\d \P\notag\\
		< & \epsilon.\notag
\end{align}
Hence the first term on the right-hand side of \eqref{equ.4.13} converges to $0$ as $k\rightarrow\infty$ by \eqref{equ.4.1} and Lemma \ref{lem.4.2}.
The second term also converges to $0$ as $k\rightarrow\infty$ by the boundedness of $DA$ and $DA_{k}$, condition \eqref{equ.4.16} and the Lebesgue dominated convergence theorem. Therefore, the first term of \eqref{equ.4.14} converges to $0$ as $k\rightarrow\infty$. The second term has a similar estimate, so we get
\begin{equation}\label{equ.4.18}
\lim_{k\rightarrow\infty} \int_{X}\int_{\Omega}\vert D\Phi_{k}-D\Phi\vert^{2}\d \P\d\mu=0.
\end{equation}

By $\Vert A\Vert^{-1}\leq \Vert A^{-1}\Vert$ for any invertible matrix $A$ and Lemma \ref{lem.4.2},
$$\int_{X}\int_{\Omega}\Vert D\Phi_{k}\Vert^{-2}\d \P\d\mu_{k}\leq\int_{X}\int_{\Omega}\Vert (D\Phi_{k})^{-1}\Vert^{2}\d \P\d\mu_{k}< +\infty.$$
According to \eqref{equ.4.9}, \eqref{equ.4.26} and \eqref{equ.4.18},
$$
\lim_{k\rightarrow\infty}\int_{X}\int_{\Omega}\ln\Vert\wedge^{l} (D\Phi)\Vert\d \P\d\mu_{k}-	\int_{X}\int_{\Omega}\ln\Vert\wedge^{l} (D\Phi_{k})\Vert\d \P\d\mu_{k}\leq 0.
$$
By the similar method,
$$
\lim_{k\rightarrow\infty}\int_{X}\int_{\Omega}\ln\Vert\wedge^{l} (D\Phi_{k})\Vert\d \P\d\mu_{k}-	\int_{X}\int_{\Omega}\ln\Vert\wedge^{l} (D\Phi)\Vert\d \P\d\mu_{k}\leq 0.
$$
Due to \eqref{equ.4.4}, the first term of inequality \eqref{term.4} converges to $0$, as $k\rightarrow\infty$.  The proof is complete.
\end{proof}

We now give some sufficient conditions to guarantee the continuity of invariant measures with respect to coefficients. Initially, we introduce the notion of uniform Lyapunov functions from Huang et al \cite{Huang1} and \cite{Huang2}.
A function $V:\R^{d}\to\R$ is called \textit{Lyapunov function} with respect to equation \eqref{equ.4.7} if $V$ is a nonnegative $C^{2}$ function such that
$$
\lim_{x\rightarrow\infty}V(x)=+\infty,
$$
and for a bounded open set $U\subset \R^{d}$ with regular boundary, there exists some constant $\gamma>0$ satisfying
$$
\mathcal{L}V (x)\leq-\gamma \quad \hbox{for } x\in \R^{d}\setminus U,
$$
where $\mathcal{L}$ is the adjoint Fokker--Planck operator generated by equation \eqref{equ.4.7}. Thereby, a $C^{2}$ function is a \textit{uniform Lyapunov function} in $\R^{d}$ with respect to equations \eqref{equ.4.8} if it is a Lyapunov function with respect to each equation \eqref{equ.4.8}, and $\gamma$ and $U$ are independent of $k$. Then we have the following continuity result on invariant measures.

\begin{prop}\label{pro.4.2}
For equations \eqref{equ.4.7} and \eqref{equ.4.8}, suppose that coefficients $A,B,A_{k},B_{k}$ are locally Lipschitz continuous, there is a uniform Lyapunov function and $B(x)B^{\intercal}(x), B_{k}(x)B_{k}^{\intercal}(x),$ $x\in U$ are invertible.  Then there exist unique invariant measures $\mu$ and $\mu_{k}$ of equations \eqref{equ.4.7} and \eqref{equ.4.8}, respectively. Furthermore, if
$$
\lim\limits_{k\rightarrow\infty}\vert A_{k}(x)-A(x)\vert+\vert B_{k}(x)-B(x)\vert=0\quad \text{point-wise},
$$
then $\mu_{k}$ converges to $\mu$
under the $\text{weak}^{\ast}$-topology.
\end{prop}

\begin{proof}
Based on Huang et al \cite{Huang1}, there exist unique invariant measures $\mu$ and $\mu_{k}$ for equations \eqref{equ.4.7} and \eqref{equ.4.8}, respectively. Furthermore, through Ma and Liu \cite{LM}, it can be shown that $\mu_{k}$ converges to $\mu$ under the $\text{weak}^{\ast}$-topology.
\end{proof}

By combining Theorem \ref{the.4.1} and Proposition \ref{pro.4.2}, we can derive the following conclusion regarding the continuity of Lyapunov exponents with respect to the coefficients.

\begin{theorem}\label{the.4.3}
For equations \eqref{equ.4.7} and \eqref{equ.4.8}, let coefficients $A,B,A_{k},B_{k}$ belong to $C^{1}(\R^{d})$, $DA,DB, DA_{k}, DB_{k}$ be bounded by a constant $K$ and  admit the same modulus $m$ of continuity. Suppose that there is a uniform Lyapunov function with respect to equations~\eqref{equ.4.7} and~\eqref{equ.4.8}, and $B(x)B^{\intercal}(x), B_{k}(x)B_{k}^{\intercal}(x),$ $x\in U$ are invertible.
Then there exist  unique invariant measures $\mu$ and $\mu_{k}$ of equations \eqref{equ.4.7} and \eqref{equ.4.8}, respectively. Furthermore, if
$$
\lim\limits_{k\rightarrow\infty}\vert A_{k}(x)-A(x)\vert+\vert B_{k}(x)-B(x)\vert=0 \quad \text{point-wise}
$$
and
$$
\lim\limits_{k\rightarrow\infty}\vert DA_{k}(x)-DA(x)\vert+\vert DB_{k}(x)-DB(x)\vert=0\quad \text{point-wise},
$$
then $\mu_{k}$ converges to $\mu$
under the $\text{weak}^{\ast}$-topology and  for any $1\leq i\leq d$,
\begin{center}
$\lim\limits_{k\rightarrow\infty}\lambda_{i,k}=\lambda_{i},$
\end{center}
where $\lambda_{i}$ and $\lambda_{i,k}$ are the Lyapunov exponents corresponding to $\mu$ and $\mu_{k}$ respectively.
\end{theorem}

\begin{remark}\rm
\begin{itemize}
\item[(1)] The uniform Lyapunov condition and non-degenerate condition are not necessary; other conditions that ensure the existence and uniqueness of invariant measures of equations \eqref{equ.4.7} and \eqref{equ.4.8} can also support Theorem \ref{the.4.3}.
\item[(2)] Consider the following $d$-dimensional SDE
\begin{equation}\label{m-dim}
\d X=A(X)\d t+\sum_{i=1}^{n}B_{i}(X)\d W_{i},\quad\quad n\in\N^{+},
\end{equation}
where $A,B_{i}:\R^{d}\rightarrow\R^{d}, 1\leq i\leq n$ and $\left\{W_{i}\right\}_{0\leq i\leq n}$ are independent Brownian motions. The conclusions  regarding equation \eqref{equ.4.7} in this section apply to equation \eqref{m-dim} as well; the proof is just a modification
of notations.
\end{itemize}
\end{remark}

\section{Lipschitz/H\"older continuity of Lyapunov exponents for autonomous SDEs}\label{sec.5}

In this section, our objective is to establish the Lipschitz/H\"older continuity of all Lyapunov exponents with respect to coefficients for SDEs \eqref{equ.4.7} and \eqref{equ.4.8}. In contrast to Theorem~\ref{the.4.3}, the results stated here require stronger dissipative conditions and a uniform Lipschitz constant for derivatives of coefficients. Numerous studies have delved into the Lipschitz or H\"older continuity of Lyapunov exponents, such as Duarte and Klein~\cite{Duarte-L}, Duarte, Klein and Poletti~\cite{Duarte}, among others. Some of these studies have provided valuable insights and recommendations for our results. For better expression, we define certain norms. For any Borel probability measure $\mu$ defined on $\R^{d}$ and differentiable mappings $\phi:\R^{d}\rightarrow\R^{d}$ with Jacobian  $D\phi:\R^{d}\rightarrow\R^{d\times d}$, the $L^{p}$ norms $(p\geq 2)$ under $\mu$ are defined as
$$
\Vert\phi\Vert_{L^{p}(\mu)}:=\Big(\int_{\R^{d}}\vert\phi(x)\vert^{p}\d\mu(x)\Big)^{\frac{1}{p}}, \quad \Vert D\phi\Vert_{L^{p}(\mu)}:=\Big(\int_{\R^{d}}\vert D\phi(x)\vert^{p}\d\mu(x)\Big)^{\frac{1}{p}}.
$$
Then we extend this to the $L^{p,p}$ norm as
$$
\Vert\phi\Vert_{L^{p,p}(\mu)}:=\Vert\phi\Vert_{L^{p}(\mu)}+\Vert D\phi\Vert_{L^{p}(\mu)}.
$$

Next, we will establish the Lipschitz continuity of solution $\Phi(x,t),t\geq0$ and its derivatives with respect to the coefficients of equations \eqref{equ.4.7} and \eqref{equ.4.8}.

\begin{lemma}\label{lem.5.2}
For SDEs \eqref{equ.4.7} and \eqref{equ.4.8}, assume that coefficients $A,B,A_{k},B_{k}$ belong to $ C^{1,1}(\R^{d})$, $DA,DB,DA_{k},DB_{k}$ are bounded by a constant $K$ and have Lipschitz constants $L,L_{k}$, respectively. Suppose that there exists an  invariant measure $\mu$ of SDE \eqref{equ.4.7}. Then for any $p\geq2$ and $0\leq t\leq T$ with $T\in [0,+\infty)$,
\begin{equation}\label{equ.4.5}
\int_{X}E\vert \Phi_{k}(x,t)-\Phi(x,t)\vert^{p}\d\mu(x)\leq C_{1}\big(\Vert A_{k}-A\Vert^{p}_{L^{p}(\mu)}+\Vert B_{k}-B\Vert^{p}_{L^{p}(\mu)}\big),
\end{equation}
where $C_{1}=C_{1}(p,K,T)>0$ is a constant; and for $1\le l \le d$
\begin{equation}\label{equ.4.6}
\int_{X}E\vert \partial_{l}\Phi_{k}(x)-\partial_{l}\Phi(x)\vert^{2}\d\mu(x)\leq C_{2}\big(\Vert A_{k}-A\Vert^{2}_{L^{p,p}(\mu)}+\Vert B_{k}-B\Vert^{2}_{L^{p,p}(\mu)}\big),
\end{equation}
where $p>2$ and $C_{2}=C_{2}(p,L_{k},K)>0$ is a constant.
\end{lemma}

\begin{proof}
\textbf{(1)}
In this part, we prove the inequality \eqref{equ.4.5} for the case when $p=2$. The corresponding $p>2$ follows by analogous reasoning.
Note that for any $0\leq t\leq T$,
$$\begin{aligned}
\int_{X}E\vert \Phi_{k}(x,t)- \Phi(x,t)\vert^{2}\d\mu\leq& 3\int_{X}E\vert\eta_{k}(t)\vert^{2}\d\mu\\
&+(3TK^{2}+3K^{2}) \int_{X}\int_{0}^{t} E\vert \Phi_{k}(x,s)- \Phi(x,s)\vert^{2}\d s\d\mu,
\end{aligned}
$$
where
$$
\begin{aligned}
\eta_{k}(t)=\int_{0}^{t}A_{k}(\Phi(x,s))-A(\Phi(x,s))\d s+\int_{0}^{t}B_{k}(\Phi(x,s))-B(\Phi(x,s))\d W.
\end{aligned}
$$
The Gronwall inequality implies that
\begin{align}\label{equ.5.12}
\int_{X}E\vert \Phi_{k}(x,t)- \Phi(x,t)\vert^{2}\d\mu
\leq 3e^{3T^{2}K^{2}+3TK^{2}} \int_{X}E\vert\eta_{k}(t)\vert^{2}\d\mu.
\end{align}
Next, we have
\begin{align}\label{equ.5.13}
\int_{X}E\vert\eta_{k}(t)\vert^{2}\d\mu
\leq & 2T\int_{X}\int_{0}^{t}E\vert A_{k}(\Phi(x,s))-A(\Phi(x,s))\vert^{2}\d s\d\mu\notag\\
&+2\int_{X}\int_{0}^{t}E\vert B_{k}(\Phi(x,s))-B(\Phi(x,s))\vert^{2}\d s\d\mu\notag\\
=&2T\int_{X}\int_{0}^{t}E\vert A_{k}(x)-A(x)\vert^{2}\d s\d\mu+2\int_{X}\int_{0}^{t}E\vert B_{k}(x)-B(x)\vert^{2}\d s\d\mu\notag\\
\leq&C(T)\big(\Vert A_{k}-A\Vert^{2}_{L^{2}(\mu)}+\Vert B_{k}-B\Vert^{2}_{L^{2}(\mu)}\big)
\end{align}
by the Cauchy-Schwarz inequality, It\^o's  isometry and invariance of measure $\mu$, where $C(T)$ is a constant dependent only on $T$. According to \eqref{equ.5.12} and \eqref{equ.5.13}, we get inequality \eqref{equ.4.5} for $p=2$.

\textbf{(2)} In the second part, we prove the inequality \eqref{equ.4.6}. In Theorem \ref{the.4.1}, the inequality \eqref{equ.4.14} holds. For the first term of \eqref{equ.4.14}, by extending the estimation of \eqref{equ.4.13}, we obtain that for $p>2$
\begin{align}
&\Big( \int_{X}\int_{0}^{1}E\vert  DA_{k}(\Phi_{k}(x,t))-DA_{k}(\Phi(x,t))\vert^{p} \d t\d\mu\Big) ^{\frac{2}{p}}\notag\\
&+\Big( \int_{X}\int_{0}^{1}E\vert DA_{k}(\Phi(x,t))-DA(\Phi(x,t))\vert^{p} \d t\d\mu\Big) ^{\frac{2}{p}}\notag\\
\leq &L_{k}^{2}\Big( \int_{X}\int_{0}^{1}E\vert  \Phi_{k}(x,t)-\Phi(x,t)\vert^{p} \d t\d\mu\Big) ^{\frac{2}{p}}+\Big( \int_{X}\int_{0}^{1}E\vert DA_{k}(x)-DA(x)\vert^{p} \d t\d\mu\Big) ^{\frac{2}{p}}\notag\\
\leq& C(p,L_{k})\big(\Vert A_{k}-A\Vert^{2}_{L^{p,p}(\mu)}+\Vert B_{k}-B\Vert^{2}_{L^{p}(\mu)}\big)\notag
\end{align}
holds by Lipschitz continuity of $DB_{k}$, \eqref{equ.4.5} and invariance of measure $\mu$.
Similarly, for the second term, we have
\begin{align}
&\Big( \int_{X}\int_{0}^{1}E\vert  DB_{k}(\Phi_{k}(x,t))-DB_{k}(\Phi(x,t))\vert^{p} \d t\d\mu\Big) ^{\frac{2}{p}}\notag\\
&+\Big( \int_{X}\int_{0}^{1}E\vert DB_{k}(\Phi(x,t))-DB(\Phi(x,t))\vert^{p} \d t\d\mu\Big) ^{\frac{2}{p}}\notag\\
\leq& C(p,L_{k})\big(\Vert A_{k}-A\Vert^{2}_{L^{p}(\mu)}+\Vert B_{k}-B\Vert^{2}_{L^{p,p}(\mu)}\big).\notag
\end{align}
According to \eqref{equ.4.14}, \eqref{equ.4.13} and
$$
\int_{0}^{1}E\vert\partial_{l}\Phi(x,t)\vert^{q}\d t\leq C(p,K)
$$
by Lemma~\ref{lem.4.2} with $\frac{2}{p}+\frac{2}{q}=1$, we get by H\"older's inequality
$$
\begin{aligned}
\int_{X}E\vert\partial_{l}\Phi_{k}(x)-\partial_{l}\Phi(x)\vert^{2}\d\mu\leq C(p,L_{k},K)\big(\Vert A_{k}-A\Vert^{2}_{L^{p,p}(\mu)}+\Vert B_{k}-B\Vert^{2}_{L^{p,p}(\mu)}\big).
\end{aligned}
$$
Hence, the proof is complete.
\end{proof}

Since the Lipchitz continuity of $\partial_{l}\Phi$ is well-known, we state the result without proof.
\begin{lemma}\label{lem.5.3}
For SDE \eqref{equ.4.7}, assume that coefficients $A,B$ belong to $ C^{1,1}(\R^{d})$,   $DA,DB$ are bounded by a constant $K$ and their Lipschitz constants are bounded by $L$.   Then for any $x,y\in \R^{d}$,
$$
E\vert \partial_{l}\Phi(x)-\partial_{l}\Phi(y)\vert^{2}\leq C \vert x-y\vert^{2},
$$
where $C=C(L,K)>0$ is a constant.
\end{lemma}

By the similar method stated by Da Prato and Zabczyk \cite[p. 106]{Prato}, we can get the following result:
\begin{lemma}\label{lem.5.1}
For SDEs \eqref{equ.4.7} and \eqref{equ.4.8}, if the strict monotonicity condition holds, i.e. the coefficients $A,B,A_{k},B_{k}$ satisfy for any $x,y\in\R^{d}$ and some constant $\lambda>0$
\begin{equation}\label{equ.5.8}
2\langle x-y, A(x)-A(y)\rangle +\vert B(x)-B(y)\vert^{2}\leq-\lambda\vert x-y\vert^{2},
\end{equation}
and
\begin{equation}\label{equ.5.30}
2\langle x-y, A_{k}(x)-A_{k}(y)\rangle +\vert B_{k}(x)-B_{k}(y)\vert^{2}\leq-\lambda\vert x-y\vert^{2},
\end{equation}
then solutions $\Phi(x,t)$ and $\Phi_{k}(x,t)$ have the exponentially contraction property, i.e. for any $x,y\in\R^{d}$,
$$
E\vert \Phi(x,t)-\Phi(y,t)\vert^{2} \leq e^{-\lambda t}\vert x-y\vert^{2}, \quad E\vert \Phi_{k}(x,t)-\Phi_{k}(y,t)\vert^{2} \leq e^{-\lambda t}\vert x-y\vert^{2}.
$$
\end{lemma}

\begin{theorem}\label{the.5.1}
For SDEs \eqref{equ.4.7} and \eqref{equ.4.8}, assume that coefficients $A,B,A_{k},B_{k}$ belong to $ C^{1,1}(\R^{d})$, $DA,DB,DA_{k},DB_{k}$ are bounded by a constant $K$ and their Lipschitz constants are bounded by $L$.
If \eqref{equ.5.8} and \eqref{equ.5.30} hold for any $x,y\in\R^{d}$ and some constant $\lambda>0$,
then there exist unique invariant measures $\mu,\mu_{k},k\in\N^{+}$ of equations \eqref{equ.4.7} and \eqref{equ.4.8} respectively and for any $p>2$
\begin{equation}\label{equ.4.29}
\vert\lambda_{i,k}-\lambda_{i}\vert\leq C\big(\Vert A_{k}-A\Vert_{L^{p,p}(\mu)}+\Vert B_{k}-B\Vert_{L^{p,p}(\mu)}\big) \quad \text{for any } 1\leq i \leq d,
\end{equation}
where $C=C(p,L,K)$ is a positive constant, and  $\lambda_{i}$ and $\lambda_{i,k}$ are Lyapunov exponents corresponding to $\mu$ and $\mu_{k}$.
\end{theorem}

\begin{proof}
By Da Prato and Zabczyk \cite[p. 105]{Prato}, if coefficients of equation~\eqref{equ.4.7} satisfy condition \eqref{equ.5.8}, then there exists a unique invariant Borel probability measure $\mu$. The measure $\mu$ is approximated by the $n$-step transition probability $P(x,n,\cdot)$ of the Markov process $\Phi(x,t), t\in\R^{+}$, i.e. for any $x\in X$ and any Lipschitz continuous bounded function $\phi$,
\begin{equation}\label{equ.5.9}
\lim\limits_{n\rightarrow\infty}\int_{X}\phi(y)P(x,n,\d y)=\int_{X}\phi(y)\d\mu(y).
\end{equation}
Similarly, it follows from condition \eqref{equ.5.30} that for any $k\in\N^{+}$,
\begin{equation}\label{equ.5.10}
\lim\limits_{n\rightarrow\infty}\int_{X}\phi(y)P_{k}(x,n,\d y)=\int_{X}\phi(y)\d\mu_{k}(y),
\end{equation}
where $P_{k}(x,n,\cdot)$ is the $n$-step transition probability of $\Phi_{k}(x,t)$.

Note that
\begin{align}\label{equ.5.5}
\Big\vert\sum_{i=1}^{l}\lambda_{i,k}-\sum_{i=1}^{l}\lambda_{i}\Big\vert=&\Big\vert\int_{X}\int_{\Omega}\ln\Vert \wedge^{l}D\Phi_{k}\Vert\d \P\d\mu_{k}-\int_{X}\int_{\Omega}\ln\Vert \wedge^{l}D\Phi\Vert\d \P\d\mu\Big\vert\notag\\
\leq& \Big	\vert
\int_{X}\int_{\Omega}\ln\Vert \wedge^{l}D\Phi_{k}\Vert\d \P\d\mu_{k}-\int_{X}\int_{\Omega}\ln\Vert \wedge^{l}D\Phi_{k}\Vert\d \P\d\mu\Big\vert\\
&+\Big	\vert\int_{X}\int_{\Omega}\ln\Vert \wedge^{l}D\Phi_{k}\Vert\d \P\d\mu-\int_{X}\int_{\Omega}\ln\Vert \wedge^{l}D\Phi\Vert\d \P\d\mu\Big\vert.\notag
\end{align}
For the first term on the right-hand side of \eqref{equ.5.5}, we have
\begin{align}\label{equ.5.14}
&\Big\vert\int_{X}\int_{\Omega}\ln\Vert \wedge^{l}D\Phi_{k}\Vert\d \P\d\mu_{k}-\int_{X}\int_{\Omega}\ln\Vert \wedge^{l}D\Phi_{k}\Vert\d \P\d\mu\Big\vert\notag\\
\leq& \Big\vert\int_{X}\int_{\Omega}\ln\Vert \wedge^{l}D\Phi_{k}\Vert\d \P\d\mu_{k}-\int_{X}\int_{X}\int_{\Omega}\ln\Vert \wedge^{l}D\Phi_{k}(y)\Vert\d \P P_{k}(x,n,\d y)\d\mu(x)\Big\vert\\
&+\Big\vert\int_{X}\int_{X}\int_{\Omega}\ln\Vert \wedge^{l}D\Phi_{k}(y)\Vert\d \P P_{k}(x,n,\d y)\d\mu(x)\notag\\
&\qquad\qquad\qquad-\int_{X}\int_{X}\int_{\Omega}\ln\Vert \wedge^{l}D\Phi_{k}(s)\Vert\d \P P(x,n,\d s)\d\mu(x)\Big\vert\notag\\
&+\Big\vert\int_{X}\int_{X}\int_{\Omega}\ln\Vert \wedge^{l}D\Phi_{k}(s)\Vert\d \P P(x,n,\d s)\d\mu(x)-\int_{X}\int_{\Omega}\ln\Vert \wedge^{l}D\Phi_{k}\Vert\d \P\d\mu\Big\vert.\notag
\end{align}
By \eqref{equ.5.9}, \eqref{equ.5.10} and the Lebesgue dominated convergence theorem, the first and third terms of the above formula converge to $0$ as $n\rightarrow+\infty$. For the second term, we will prove the following inequality
\begin{align}\label{equ.5.1}
\Big\vert\int_{X }\Big[\int_{X}\int_{\Omega}\ln\Vert \wedge^{l}&D\Phi_{k}(y)\Vert\d \P P_{k}(x,n,\d y)-\int_{X}\int_{\Omega}\ln\Vert \wedge^{l}D\Phi_{k}(s)\Vert\d \P P(x,n,\d s)\Big]\d\mu(x)\Big\vert\notag\\
\leq& C\sum_{i=1}^{n}e^{-\lambda (i-1)}\big(\Vert A_{k}-A\Vert_{L^{2}(\mu)}+\Vert B_{k}-B\Vert_{L^{2}(\mu)}\big).
\end{align}

Firstly, we prove \eqref{equ.5.1} for the case when $n=1$. The following estimate
\begin{align}
\Big\vert\int_{X}\Big[	\int_{X}&\int_{\Omega}\ln\Vert \wedge^{l} D\Phi_{k}(y)\Vert\d \P P_{k}(x,1,\d y)-\int_{X}\int_{\Omega}\ln\Vert\wedge^{l} D\Phi_{k}(s)\Vert\d \P P(x,1,\d s)\Big]\d\mu(x)\Big\vert\notag\\
\leq &C\times \Big[\int_{X} \int_{X\times X}E\vert D\Phi_{k}(y)-D\Phi_{k}(s)\vert^{2} P_{k}(x,1,\d y)\times P(x,1,\d s)\d\mu(x)\Big]^{\frac{1}{2}}\notag\\
&\quad\times \Big( \int_{X}\int_{X}E\Vert (D\Phi_{k}(s))^{-1}\Vert^{2} P(x,1,\d s)\d\mu(x)\Big)^{\frac{1}{2}}\notag\\
\leq&C\Big[\int_{X}\int_{X\times X}\vert y-s\vert^{2} P_{k}(x,1,\d y)\times P(x,1,\d s)\d\mu(x)\Big]^{\frac{1}{2}}\notag\\
=&C\Big[\int_{X}E\vert\Phi_{k}(x)-\Phi(x)\vert^{2}\d\mu(x)\Big]^{\frac{1}{2}}\notag\\
\leq &C\big(\Vert A_{k}-A\Vert_{L^{2}(\mu)}+\Vert B_{k}-B\Vert_{L^{2}(\mu)}\big)\notag
\end{align}
holds by the similar method to that of \eqref{equ.4.9}, Lemmas \ref{lem.5.3}, \ref{lem.4.2} and \ref{lem.5.2}. The case for $n=1$ is valid.

Next, consider \eqref{equ.5.1} for the case when $n=2$. We can get
\begin{align}
&\Big\vert\int_{X}\Big[\int_{X}E(\ln\Vert \wedge^{l}D\Phi_{k}(z)\Vert) P_{k}(x,2,\d z)-\int_{X}E(\ln\Vert \wedge^{l}D\Phi_{k}(t)\Vert) P(x,2,\d t)\Big]\d\mu(x)\Big\vert\notag\\
=&\Big\vert\int_{X}\Big[\int_{X}\int_{X}E(\ln\Vert \wedge^{l}D\Phi_{k}(z)\Vert) P_{k}(y,1,\d z) P_{k}(x,1,\d y)\notag\\
&\qquad-\int_{X}\int_{X}E(\ln\Vert \wedge^{l}D\Phi_{k}(t)\Vert) P(s,1,\d t)P(x,1,\d s)\Big]\d\mu(x)\Big\vert\notag\\
\leq&\Big\vert\int_{X}\Big[\int_{X}\int_{X}E(\ln\Vert \wedge^{l}D\Phi_{k}(z)\Vert) P_{k}(y,1,\d z) P_{k}(x,1,\d y)\label{equ.5.11}\\
&\qquad-\int_{X }\int_{X}E(\ln\Vert \wedge^{l}D\Phi_{k}(t)\Vert) P_{k}(s,1,\d t) P(x,1,\d s)\Big]\d\mu(x)\Big\vert\notag\\
&+\Big\vert\int_{X}\Big[\int_{X}\int_{X}E(\ln\Vert \wedge^{l}D\Phi_{k}(n)\Vert) P_{k}(m,1,\d n) P(x,1,\d m)\notag\\
&\qquad-\int_{X}\int_{X}E(\ln\Vert \wedge^{l}D\Phi_{k}(t)\Vert) P(s,1,\d t) P(x,1,\d s)\Big]\d\mu(x)\Big\vert.\notag
\end{align}
For the first term on the right-hand side of formula \eqref{equ.5.11}, we have
\begin{align}
&\Big\vert\int_{X}\Big[\int_{X}\int_{X}E(\ln\Vert \wedge^{l}D\Phi_{k}(z)\Vert) P_{k}(y,1,\d z)P_{k}(x,1,\d y)\notag\\
&\quad\quad -\int_{X}\int_{X}E(\ln\Vert \wedge^{l}D\Phi_{k}(t)\Vert) P_{k}(s,1,\d t)P(x,1,\d s)\Big]\d\mu(x)\Big\vert\notag\\
\leq&C\times\Big[\int_{X}\int_{X\times X}E\vert D\Phi_{k}(\Phi_{k}(y))-D\Phi_{k}(\Phi_{k}(s))\vert^{2} P_{k}(x,1,\d y)\times P(x,1,\d s)\d\mu(x)\Big]^{\frac{1}{2}}\label{equ.5.3}\notag\\
&\qquad\times \Big[\int_{X}\int_{X}\int_{X}E\Vert (D\Phi_{k}(t))^{-1}\Vert^{2} P_{k}(s,1,\d t) P(x,1,\d s)\d\mu(x)\Big]^{\frac{1}{2}}\notag\\
\leq &C\Big[\int_{X}\int_{X\times X}E\vert \Phi_{k}(y)-\Phi_{k}(s)\vert^{2} P_{k}(x,1,\d y)\times P(x,1,\d s)\d\mu(x)\Big]^{\frac{1}{2}}\notag\\
=&Ce^{-\lambda}\Big[\int_{X}E\vert \Phi_{k}(x)-\Phi(x)\vert^{2}\d\mu(x)\Big]^{\frac{1}{2}}\notag\\
\leq&Ce^{-\lambda}\big(\Vert A_{k}-A\Vert_{L^{2}(\mu)}+\Vert B_{k}-B\Vert_{L^{2}(\mu)}\big)
\end{align}
by the similar method to that of \eqref{equ.4.9}, Lemmas \ref{lem.5.3}, \ref{lem.4.2}, \ref{lem.5.1} and \ref{lem.5.2}. The second term has the following estimate
\begin{align}
&\Big\vert\int_{X}\Big[	\int_{X}\int_{X}E(\ln\Vert \wedge^{l}D\Phi_{k}(n)\Vert) P_{k}(m,1,\d n) P(x,1,\d m)\notag\\
&\qquad -\int_{X}\int_{X}E(\ln\Vert \wedge^{l}D\Phi_{k}(t)\Vert) P(s,1,\d t) P(x,1,\d s)\Big]\d\mu(x)\Big\vert\notag\\
\leq &C\times \Big[\int_{X} \int_{X \times X}E\vert D\Phi_{k}(\Phi_{k}(m))-D\Phi_{k}(\Phi(s))\vert^{2} P(x,1,\d m)\times P(x,1,\d s)\d\mu(x)\Big]^{\frac{1}{2}}\notag\\
&\qquad\times \Big[ \int_{X}\int_{X}\int_{X}E\Vert (D\Phi_{k}(t))^{-1}\Vert^{2} P(s,1,\d t)P(x,1,\d s)\d\mu(x)\Big]^{\frac{1}{2}}\notag\\
\leq& C\Big[\int_{X}\int_{X\times X}E\vert\Phi_{k}(m)-\Phi(s)\vert^{2} P(x,1,\d m)\times P(x,1,\d s)\d\mu(x)\Big]^{\frac{1}{2}}\label{equ.5.2}\notag\\
=&C\Big(\int_{X}E\vert\Phi_{k}(x)-\Phi(x)\vert^{2}\d\mu(x)\Big)^{\frac{1}{2}}\notag\\
\leq& C\big(\Vert A_{k}-A\Vert_{L^{2}(\mu)}+\Vert B_{k}-B\Vert_{L^{2}(\mu)}\big)
\end{align}
by the similar method to that of \eqref{equ.4.9}, Lemmas \ref{lem.5.3} and \ref{lem.4.2}, the invariance of measure $\mu$ and Lemma \ref{lem.5.2}.
According to \eqref{equ.5.3} and \eqref{equ.5.2}, we complete the proof of inequality \eqref{equ.5.1} for the case when $n=2$.
By induction, we can prove that inequality \eqref{equ.5.1} is true for any $n\in\N^{+}$. It follows from \eqref{equ.5.14} that
\begin{align}\label{equ.5.7}
&\Big\vert	\int_{X}\int_{\Omega}\ln\Vert \wedge^{l}D\Phi_{k}\Vert\d \P\d\mu_{k}-\int_{X}\int_{\Omega}\ln\Vert\wedge^{l} D\Phi_{k}\Vert\d \P\d\mu\Big\vert\notag\\
\leq&  C\lim_{n\rightarrow \infty}\sum_{i=1}^{n}e^{-\lambda (i-1)}\big(\Vert A_{k}-A\Vert_{L^{2}(\mu)}+\Vert B_{k}-B\Vert_{L^{2}(\mu)}\big)\notag\\
\leq& C\big(\Vert A_{k}-A\Vert_{L^{2}(\mu)}+\Vert B_{k}-B\Vert_{L^{2}(\mu)}\big).
\end{align}

For the second term of \eqref{equ.5.5}, according to the similar method to that of \eqref{equ.4.9}, Lemmas \ref{lem.4.2} and \ref{lem.5.2}, we know that
\begin{align}\label{equ.5.6}
&\Big\vert\int_{X}E(\ln\Vert \wedge^{l}D\Phi_{k}\Vert-\ln\Vert \wedge^{l}D\Phi\Vert)\d \mu\Big\vert\notag\\
\leq&C\Big ( \int_{X}E\vert D\Phi_{k}-D\Phi\vert^{2}\d\mu\Big)^{\frac{1}{2}}\notag\\
\leq & C\big(\Vert A_{k}-A\Vert_{L^{p,p}(\mu)}+\Vert B_{k}-B\Vert_{L^{p,p}(\mu)}\big).
\end{align}

By \eqref{equ.5.5}, \eqref{equ.5.7} and \eqref{equ.5.6}, the proof is complete.
\end{proof}

\begin{remark}\rm\label{rem.5.1}
\begin{enumerate}
\item[(1)] In Theorem \ref{the.5.1}, we request the monotonicity condition for SDEs \eqref{equ.4.7} and \eqref{equ.4.8} to get the Lipschitz continuity of Lyapunov exponents; this condition can be weakened. Indeed, we only need that the distributions of solutions $\Phi(x,t)$ and $\Phi_{k}(x,t)$
    weakly converge to invariant measures not too slowly, i.e. \eqref{equ.5.9}, \eqref{equ.5.10} and \eqref{equ.5.7} hold.
\item[(2)] Clearly,  it follows from \eqref{equ.4.29} that under the same condition of Theorem \ref{the.5.1}, we have
$$
\vert\lambda_{i,k}-\lambda_{i}\vert\leq C\big( \Vert A_{k}-A\Vert_{C^{1}}+\Vert B_{k}-B\Vert_{C^{1}}\big) \quad 	\text{for any } 1\leq i \leq d.
$$
\item[(3)] By Lemmas \ref{lemA1} and \ref{lemA2} in the appendix, and the similar method to that of Theorem \ref{the.5.1}, it is not difficult to prove the $\alpha$-H$\ddot{\mathrm{o}}$lder continuity ($0<\alpha<1$) of Lyapunov exponents under weaker regularity conditions:

If coefficients $A,B,A_{k},B_{k},k\in \N^{+}$ are $C^{1,\alpha}$ functions,  the H\"older-$\alpha$-seminorms of $DA,DB,DA_{k},DB_{k}$ are
bounded a constant $L$ and the other conditions remain the same, then for any $p>2$
\begin{align}
\qquad\qquad\vert\lambda_{i,k}-\lambda_{i}\vert\leq& C\big( \Vert A_{k}-A\Vert_{L^{2}(\mu)}^{\alpha}+\Vert B_{k}-B\Vert_{L^{2}(\mu)}^{\alpha}\notag\\
&+\Vert DA_{k}-DA\Vert_{L^{p}(\mu)}+\Vert DB_{k}-DB\Vert_{L^{p}(\mu)}\big)\quad \text{for any } 1\leq i\leq d,\notag
\end{align}
where $C=C(\alpha,p,L,K)$ is a positive constant.
\item[(4)] The conclusions drawn in this section regarding equation \eqref{equ.4.7} also apply to equation \eqref{m-dim}.
\end{enumerate}
\end{remark}

Next, by Theorem \ref{the.5.1}, we can establish a local Lipschitz continuity theory for Lyapunov exponents on $C^{1,1}(\R^{d})$. Consider the following two SDEs defined on $\R^{d}$
\begin{equation}\label{equ.4.31}
\d X= A_{1}(X)\d t+B_{1}(X)\d W
\end{equation}
and
\begin{equation}\label{equ.4.32}
\d X= A_{2}(X)\d t+B_{2}(X)\d W,
\end{equation}
where $A_{1},B_{1},A_{2},B_{2}\in C^{1,1}(\R^{d})$. We denote the $i$-th largest Lyapunov exponent of equations \eqref{equ.4.31} and \eqref{equ.4.32} as $\lambda_{i}(A_{1},B_{1})$ and $\lambda_{i}(A_{2},B_{2})$ for any $1\leq i\leq d$, respectively. For any $F\in C^{1,1}(\R^{d})$, we define a new norm
$$
\Vert F\Vert_{C^{1}_{Lip}}:=\Vert F\Vert_{C^{1}}+\sup_{\substack{x,y\in\R^{d} \\ x\neq y}}\dfrac{\vert DF(x)-DF(y)\vert}{\vert x-y\vert}.
$$
Then the following corollary holds.
\begin{coro}
For SDE \eqref{equ.4.7}, if coefficients $A,B\in C^{1,1}(\R^{d})$ satisfy
$$
2\langle x-y, A(x)-A(y)\rangle +\vert B(x)-B(y)\vert^{2}\leq-\lambda\vert x-y\vert^{2},
$$
for any $x,y\in\R^{d}$ with some constant $\lambda>0$ and $\vert DA\vert,\vert DB\vert$ are bounded by a constant $K$, then ther exist two neighborhoods $G_{1}$ of $A$ and $G_{2}$ of $B$ under norm $\Vert \cdot\Vert_{C^{1}_{Lip}}$ such that for any $A_{1},A_{2}\in G_{1}$ and $B_{1},B_{2}\in G_{2}$, equations \eqref{equ.4.31} and \eqref{equ.4.32} have unique invariant measures $\mu,\mu_{k},k\in\N^{+}$ and the corresponding Lyapunov exponents satisfy
$$
\vert \lambda_{i}(A_{1},B_{1})-\lambda_{i}(A_{2},B_{2})\vert\leq C\big( \Vert A_{1}-A_{2}\Vert_{C^{1}}+\Vert B_{1}-B_{2}\Vert_{C^{1}}\big) \quad \text{for any } 1\leq i\leq d,
$$
with some constant $C=C(L,K,G_{1},G_{2})$ dependent on the choices of $G_{1}$ and $G_{2}$.
\end{coro}

\section*{Acknowledgements}
This work is supported by National Key R\&D Program of China (No. 2023YFA1009200), NSFC (Grant 11925102), and Liaoning Revitalization Talents Program (Grant XLYC2202042).

\section*{Data availability}
Data sharing is not applicable to this article as no datasets were generated or analyzed during the current study.

\section*{Appendix}
\renewcommand{\theequation}{A.\arabic{equation}}
\setcounter{equation}{0}  
In the appendix we study the properties of derivatives for solution $\Phi(x,t),t\geq 0$ to equation \eqref{equ.4.7} when coefficients $A,B\in C^{1,\alpha}(\R^{d})$ with $0<\alpha<1$.
We establish the H\"older continuity of the derivatives with respect to the coefficients and initial values, which are key to proving the $\alpha$-H\"older continuity of the Lyapunov exponents:
\begin{lemmaA}\label{lemA1}
For SDEs \eqref{equ.4.7} and \eqref{equ.4.8}, assume that coefficients $A,B,A_{k},B_{k}$ belong to $ C^{1,\alpha}(\R^{d})$ $(0<\alpha<1)$, $DA,DB,DA_{k},DB_{k}$ are bounded by a constant $K$ and their H\"older-$\alpha$-seminorms are constants $L,L_{k}$, respectively. Suppose that there exists an  invariant measure $\mu$ of SDE \eqref{equ.4.7}. Then for any $1\le l \le d$ and $p>2$
\begin{align}
\int_{X}E\vert \partial_{l}\Phi_{k}(x)-\partial_{l}\Phi(x)\vert^{2}\d\mu(x)\leq& C\big(\Vert A_{k}-A\Vert^{2\alpha}_{L^{2}(\mu)}+\Vert B_{k}-B\Vert^{2\alpha}_{L^{2}(\mu)}\notag\\
&+\Vert DA_{k}-DA\Vert^{2}_{L^{p}(\mu)}+\Vert DB_{k}-DB\Vert^{2}_{L^{p}(\mu)}\big),\notag
\end{align}
where $C=C(\alpha,p,L_{k},K)>0$ is a constant.
\end{lemmaA}
\begin{proof}
Firstly, under this condition, the inequality \eqref{equ.4.5} holds by the proof given in \textbf{\text{(1)}} of Lemma \ref{lem.5.2}. Next, consider the first term on the right-hand side of \eqref{equ.4.14}. Note that for any $0<\alpha<1$ and $p>2$
\begin{align}
	&\int_{X}	E\Big\vert\int_{0}^{1}[DA_{k}(\Phi_{k}(x,t))-DA(\Phi(x,t))]\partial_{l}\Phi(x,t)\d t\Big\vert^{2}\d\mu\notag \\
	\leq& 2\int_{X}\int_{0}^{1}E\big\vert[DA_{k}(\Phi_{k}(x,t))-DA_{k}(\Phi(x,t))]\partial_{l}\Phi(x,t)\big\vert^{2} \d t\d\mu\notag \\
	&+2\int_{X}\int_{0}^{1}E\big\vert[DA_{k}(\Phi(x,t))-DA(\Phi(x,t))]\partial_{l}\Phi(x,t)\big\vert^{2} \d t\d\mu\notag \\
	\leq& C(\alpha,K)\Big( \int_{X}\int_{0}^{1}E\vert  DA_{k}(\Phi_{k}(x,t))-DA_{k}(\Phi(x,t))\vert^{\frac{2}{\alpha}} \d t\d\mu\Big) ^{\alpha}\notag\\
	&+C(p,K)\Big( \int_{X}\int_{0}^{1}E\vert DA_{k}(\Phi(x,t))-DA(\Phi(x,t))\vert^{p} \d t\d\mu\Big) ^{\frac{2}{p}}\notag
\end{align}
holds by the Cauchy-Schwarz inequality, H\"older's inequality and Lemma \ref{lem.4.2}.
By the similar method to that of \textbf{\text{(2)}} in the proof of Lemma \ref{lem.5.2}, we know that
\begin{align}
&\Big( \int_{X}\int_{0}^{1}E\vert  DA_{k}(\Phi_{k}(x,t))-DA_{k}(\Phi(x,t))\vert^{\frac{2}{\alpha}} \d t\d\mu\Big) ^{\alpha}\notag\\	
&\qquad+\Big( \int_{X}\int_{0}^{1}E\vert DA_{k}(\Phi(x,t))-DA(\Phi(x,t))\vert^{p} \d t\d\mu\Big) ^{\frac{2}{p}}\notag\\
\leq &L_{k}^{2}\Big( \int_{X}\int_{0}^{1}E\vert  \Phi_{k}(x,t)-\Phi(x,t)\vert^{2} \d t\d\mu\Big) ^{\alpha}+\Big( \int_{X}\int_{0}^{1}E\vert DA_{k}(x)-DA(x)\vert^{p} \d t\d\mu\Big) ^{\frac{2}{p}}\notag\\
\leq& C(\alpha,L_{k},K)\big(\Vert A_{k}-A\Vert^{2\alpha}_{L^{2}(\mu)}+\Vert B_{k}-B\Vert^{2\alpha}_{L^{2}(\mu)}+\Vert DA_{k}-DA\Vert^{2}_{L^{p}(\mu)}\big).\notag
\end{align}
Similarly, for the second term of \eqref{equ.4.14}, we have
\begin{align}
&\int_{X}E\Big\vert\int_{0}^{1}[DB_{k}(\Phi_{k}(x,t))-DB(\Phi(x,t))]\partial_{l}\Phi(x,t)\d W\Big\vert^{2}\d\mu\notag\\
\leq& C(\alpha,p,L_{k},K)\big(\Vert A_{k}-A\Vert^{2\alpha}_{L^{2}(\mu)}+\Vert B_{k}-B\Vert^{2\alpha}_{L^{2}(\mu)}+\Vert DB_{k}-DB\Vert^{2}_{L^{p}(\mu)}\big).\notag
\end{align}
Then the proof is complete.
\end{proof}

The proof of the H\"older continuity of $\partial_{l}\Phi$ $(1\leq l\leq d)$ is straightforward, so we omit it.
\begin{lemmaB}\label{lemA2}
For SDE \eqref{equ.4.7}, assume that coefficients $A,B$ belong to $C^{1,\alpha}(\R^{d})$ $(0<\alpha<1)$,   $DA,DB$ are bounded by a constant $K$ and their H\"older-$\alpha$-seminorms are bounded by $L$.   Then for any $x,y\in \R^{d}$
$$
E\vert \partial_{l}\Phi(x)-\partial_{l}\Phi(y)\vert^{2}\leq C \vert x-y\vert^{2\alpha},
$$
where $C=C(\alpha,L,K)>0$ is a constant.
\end{lemmaB}


\begin{thebibliography}{10}
	
	
\bibitem{Arnold-equation}
L. Arnold,
\newblock {\em Stochastic Differential Equations: Theory and Applications.}
\newblock {Translated from the German. Wiley--Interscience [John Wiley and Sons]}, New York--London--Sydney, 1974. xvi+228 pp.


\bibitem{Arnold-system}
L. Arnold,
\newblock {\em Random Dynamical Systems.}
\newblock {	Springer Monographs in Mathematics,} Springer--Verlag, Berlin, 1998. xvi+586 pp.
	
	
\bibitem{Avila-E}	
A. Avila, A. Eskin and M. Viana, Continuity of the Lyapunov exponents of random matrix products. arXiv:2305.06009.
	
\bibitem{Avila}
A. Avila and M. Viana,
\newblock
Extremal Lyapunov exponents: an invariance principle and applications.
\newblock{\em Invent. Math.}
{\bf 181} (2010), 115--189.
	
	
\bibitem{Baxendale}
P. H. Baxendale,
\newblock
Lyapunov exponents and relative entropy for a stochastic flow of diffeomorphisms.
\newblock{\em 	Probab. Theory Related Fields}
{\bf 81} (1989), 521--554.
	
	
\bibitem{Bahnmuller}
J. Bahnm\"uller and T. Bogensch\"utz,
\newblock
A Margulis-Ruelle inequality for random dynamical systems.
\newblock{\em Arch. Math.} {\bf 64} (1995), 246--253.
	
	
\bibitem{Blumenthal2}
A. Blumenthal, J. Xue and Y. Yang,
Lyapunov exponents for random perturbations of coupled standard maps.
\textit{Comm. Math. Phys.} \textbf{389} (2022), 121--151.
	
	
\bibitem{Bochi2002}
J. Bochi,
\newblock Genericity of zero Lyapunov exponents.
\newblock{\em Ergodic Theory Dynam. Systems} {\bf 22} (2002), 1667--1696.
	
	
\bibitem{Bochi2005}
J. Bochi and M. Viana,
\newblock  The Lyapunov exponents of generic volume-preserving and symplectic maps.
\newblock{\em Ann.  of Math. (2)} {\bf 161} (2005), 1423--1485.
	
	
\bibitem{Bocker-Viana}
C. Bocker-Neto and M. Viana,
Continuity of Lyapunov exponents for random two-dimensional matrices.
\textit{Ergodic Theory Dynam. Systems} \textbf{37} (2017), 1413--1442.
	
	
\bibitem{Buzzi}
J. Buzzi, S. Crovisier and O. Sarig,
Continuity properties of Lyapunov exponents for surface diffeomorphisms.
\textit{Invent. Math.} \textbf{230} (2022), 767--849.
	
	
\bibitem{Cai-You}
A. Cai, C. Chavaudret, J. You and Q. Zhou,
Sharp H\"older continuity of the Lyapunov exponent of finitely differentiable quasi-periodic cocycles.
\textit{Math. Z.} \textbf{291} (2019), 931--958.

\bibitem{Carverhill}
A. Carverhill,
Flows of stochastic dynamical systems: ergodic theory.
\textit{Stochastics} \textbf{14} (1985), 273--317.

\bibitem{CLZ}
X. Cheng, Z. Liu and L. Zhang, Small perturbations may change the sign of Lyapunov exponents for linear SDEs.
{\it Stoch. Dyn.} \bf 22 \rm(2022), no. 8, Paper No. 2240038, 25 pp.

\bibitem{CLZ-pre}
X. Cheng, Z. Liu and L. Zhang, The multiplicative ergodic theorem for McKean-Vlasov SDEs. arXiv:2401.09702 (2024).
	
\bibitem{Cong}
N. D. Cong,
\newblock
Lyapunov spectrum of nonautonomous linear stochastic differential equations.
\newblock{\em Stoch. Dyn.} {\bf 1} (2001), 127--157.
	
	
\bibitem{Prato}
G. Da Prato and J. Zabczyk,
\newblock{\em Ergodicity for Infinite-Dimensional Systems.}
London Mathematical Society Lecture Note Series, 229. Cambridge University Press, Cambridge, 1996. xii+339 pp.


\bibitem{Duarte-L}
P. Duarte and S. Klein,
Continuity of the Lyapunov exponents for quasiperiodic cocycles.
\textit{Comm. Math. Phys.} \textbf{332} (2014), 1113--1166.
	
	
\bibitem{Duarte-K}
P. Duarte and S. Klein,
Continuity, positivity and simplicity of the Lyapunov exponents for quasi-periodic cocycles.
\textit{J. Eur. Math. Soc. (JEMS)} \textbf{21} (2019), 2051--2106.
	
	
\bibitem{Duarte}
P. Duarte, S. Klein and M. Poletti,
\newblock
H$\ddot{\mathrm{o}}$lder continuity of the Lyapunov exponents of linear cocycles over hyperbolic maps.
\newblock{\em Math. Z.}
{\bf 302} (2022), 2285--2325.
		
	
\bibitem{Friedman}
A. Friedman,
\newblock{\em Stochastic Differential Equations and Applications. Vol. 1.}
Probability and Mathematical Statistics, Academic Press, Vol. 28. New York-London, 1975. xiii+231 pp.
	
	
\bibitem{Furstenberg1960}
H. Furstenberg and H. Kesten,
\newblock Products of random matrices.
\newblock{\em Ann. Math. Statist.} {\bf 31} (1960), 457--469.
	
	
\bibitem{Furstenberg-C}
H. Furstenberg and Y. Kifer,
Random matrix products and measures on projective spaces.
\newblock{\em Israel J. Math.} \textbf{46} (1983),  12--32.
	
	

\bibitem{Huang1}
W. Huang, M. Ji, Z. Liu and Y. Yi, Steady states of Fokker-Planck equations: I. Existence. \it
J. Dynam. Differential Equations \bf27 \rm(2015), 721--742.	

\bibitem{Huang2}
W. Huang, M. Ji, Z. Liu and Y. Yi,
Concentration and limit behaviors of stationary measures.
\textit{Phys. D} \textbf{369} (2018), 1--17.

	
\bibitem{Kifer}
Y. Kifer,
\newblock{\em Ergodic Theory of Random Transformations.}
Progress in Probability and Statistics, 10. Birkh$\ddot{\mathrm{a}}$user Boston, Inc., Boston, MA, 1986. x+210 pp.
	
	
\bibitem{Kingman}
J. F. C. Kingman,
The ergodic theory of subadditive stochastic processes. \textit{J. Roy. Statist. Soc. Ser. B} \textbf{30} (1968), 499--510.
	
	
\bibitem{Kunita}
H. Kunita, Stochastic differential equations and stochastic flows of diffeomorphisms. \'Ecole d'\'et\'e de probabilit\'es de Saint-Flour, XII--1982, 143--303, \it Lecture Notes in Math., \bf 1097\rm, Springer, Berlin, 1984.
	
	
\bibitem{Kunita2019}
H. Kunita,
\newblock{\em Stochastic Flows and Jump-Diffusions.}
Probability Theory and Stochastic Modelling, 92. Springer, Singapore, 2019. xvii+352 pp.
	
	
\bibitem{Ledrappier}
F. Ledrappier and L. Young,
\newblock
Entropy formula for random transformations.
{\em Probab. Theory Related Fields} {\bf 80} (1988), 217--240.


\bibitem{Liao97}
M. Liao, Liapounov exponents of stochastic flows. \it Ann. Probab. \bf 25 \rm (1997), 1241--1256.

\bibitem{Liao00}
M. Liao, Decomposition of stochastic flows and Lyapunov exponents. \it Probab. Theory Related Fields \bf 117 \rm(2000),
589--607.


\bibitem{Liupeidong}
P. Liu and M. Qian,
\newblock
{\em Smooth Ergodic Theory of Random Dynamical Systems.}
\newblock	{Lecture Notes in Mathematics,} Vol. 1606. Springer-Verlag, Berlin, 1995. xii+221 pp.
		
	
\bibitem{Lyapunov1992}
A. M. Lyapunov,
\newblock The general problem of the stability of motion.
\newblock{\em Internat. J. Control} \textbf{55} (1992), 521--790.
	

\bibitem{LM}
J. Ma and Z. Liu, Continuous dependence for McKean-Vlasov SDEs under distribution-dependent Lyapunov conditions.
\it Discrete Contin. Dyn. Syst. Ser. S \rm to appear. Available also at: arXiv:2406.00309.

	
\bibitem{Mane1984}
R. Ma$\tilde{\mathrm{n}}\acute{\mathrm{e}}$,
\newblock{Oseledec's theorem from the generic viewpoint.}
\textit{Proc. Int. Congress of Mathematicians (Warszawa)}
\textbf{2} (1983), 1259--1276.

	
\bibitem{Oseledets1968}
\newblock{V. I. Oseledets. A multiplicative ergodic theorem: Lyapunov characteristic numbers for dynamical systems.
\newblock{\em Trans. Moscow Math. Soc.} \textbf{19} (1968), 197--231.}
	
\bibitem{Pesin}
Ya. B. Pesin,
Families of invariant manifolds that correspond to nonzero characteristic exponents.
\textit{Izv. Akad. Nauk SSSR Ser. Mat.} \textbf{40} (1976), 1332--1379.
	
\bibitem{Ruelle}
D. Ruelle,
\newblock{Ergodic Theory of Differentiable Dynamical Systems.}
\newblock{\em Inst. Hautes $\acute{E}$tudes Sci. Publ. Math.}  {\bf 50} (1979), 27--58.
	
	
\bibitem{Cantelli}
A. N. Shiryayev,
\newblock{\em Probability.} Second edition.
Graduate Texts in Mathematics, 95. Springer-Verlag, New York, 1996. xvi+623 pp.
	
	
\bibitem{Viana-B}
M. Viana,
\newblock{\em Lectures on Lyapunov Exponents}.
Cambridge Studies in Advanced Mathematics, 145. Cambridge University Press, Cambridge, 2014. xiv+202 pp.


\bibitem{Viana}
M. Viana and J. Yang,
\newblock
Continuity of Lyapunov exponents in the $C^{0}$ topology.
\newblock{\em Israel J. Math.} {\bf 229} (2019), 461--485.
	
	
\end{thebibliography}
\end{document}